\documentclass[11pt,twoside]{amsart}
\usepackage{latexsym,amssymb,amsmath}
\usepackage[colorlinks=true, linkcolor=blue, anchorcolor=black, citecolor=blue, filecolor=blue, menucolor= blue, urlcolor=black]{hyperref}

\numberwithin{equation}{section}
\def\demo{\noindent{\it Proof. }}
\newtheorem{theorem}{Theorem}[section]
\newtheorem{lemma}[theorem]{Lemma}
\newtheorem{proposition}[theorem]{Proposition}
\newtheorem{corollary}[theorem]{Corollary}
\newtheorem{conjecture}[theorem]{Conjecture}

\theoremstyle{definition}
\newtheorem{definition}[theorem]{Definition}
\newtheorem{procedure}[theorem]{Procedure}
\newtheorem{remark}[theorem]{Remark}
\newtheorem{example}[theorem]{Example}

\begin{document}


\title[Algebraic invariants of Geramita ideals]
{Generalized minimum distance functions 
and algebraic invariants of Geramita ideals}

\author[S. M. Cooper]{Susan M. Cooper}
\address{Department of Mathematics,
420 Machray Hall, 186 Dysart Road
University of Manitoba, Winnipeg, MB R3T 2N2 Canada.}
\email{susan.cooper@umanitoba.ca}

\author[A. Seceleanu]{Alexandra Seceleanu}
\address{Department of Mathematics,
203 Avery Hall, Lincoln, NE 68588.
}
\email{aseceleanu@unl.edu}

\author[S. O. Toh\u{a}neanu]{\c{S}tefan O. Toh\u{a}neanu}
\address{Department of Mathematics,
University of Idaho,
875 Perimeter Drive, MS 1103,
Moscow, ID 83844-1103.}
\email{tohaneanu@uidaho.edu}

\author[M. Vaz Pinto]{Maria Vaz Pinto}
\address{Departamento de Matem\'atica, Instituto Superior T\'ecnico, Universidade de Lisboa, Avenida Rovisco Pais, 1, 1049-001 Lisboa, Portugal.}
\email{vazpinto@math.ist.utl.pt}

\author[R. H. Villarreal]{Rafael H. Villarreal}
\address{
Departamento de
Matem\'aticas\\
Centro de Investigaci\'on y de Estudios
Avanzados del
IPN\\
Apartado Postal
14--740 \\
07000 Mexico City, CDMX.
}
\email{vila@math.cinvestav.mx}


\begin{abstract}
Motivated by notions from coding theory, we study the generalized minimum distance (GMD) function $\delta_I(d,r)$ of a graded ideal $I$ in a polynomial ring over an arbitrary field using commutative algebraic methods.
It is shown that $\delta_I$ is non-decreasing as a
function of $r$ and non-increasing as a function of $d$. For vanishing
ideals over finite fields, we  show that $\delta_I$ is
strictly decreasing as a function of $d$ until it stabilizes.
We also study algebraic invariants of Geramita ideals. Those ideals are graded, unmixed,
$1$-dimensional and their associated primes are generated by linear
forms.  
We also examine GMD functions of complete
intersections and show some special cases of two conjectures of
Toh\u{a}neanu--Van Tuyl and Eisenbud-Green-Harris.
\end{abstract}

\maketitle

\section{Introduction}\label{intro-section}

Let $K$ be any field, and let $C$ be a linear code that is the image of some $K$-linear map $K^s\longrightarrow K^n$. Suppose $G$ is the $s\times n$ matrix representing this map with respect to some chosen bases and assume that $G$ has no zero columns. By definition, the {\em minimum (Hamming) distance} of $C$ is $$\delta(C):=\min\{{\rm wt}({\bf v}) \mid {\bf v}\in C\setminus\{{\bf 0}\}\},$$ where for any vector ${\bf w}\in K^n$, the {\em weight} of ${\bf w}$, denoted ${\rm wt}({\bf w})$, is the number of nonzero entries in ${\bf w}$.  More generally, for $1\leq r\leq \dim_K(C)$, the {\em $r$-th generalized Hamming distance}, denoted $\delta_r(C)$, is defined as follows. For any subcode, i.e., linear subspace, $D\subseteq C$ define the support of $D$ to be $$\chi(D):=\{i \mid \text{there exists} \,\,\, (x_1,\ldots,x_n)\in D \mbox{ with }x_i\neq 0\}.$$ Then the $r$-th generalized Hamming distance of $C$ is
$$\delta_r(C):=\min_{D\subseteq C,\,\dim D=r}|\chi(D)|.$$
The {\it weight hierarchy\/} of $C$ is the sequence
$(\delta_1(C),\ldots,\delta_k(C))$, where $k=\dim(C)$. Observe that $\delta_1(C)$ equals the minimum distance $\delta(C)$. The study of these weights is related to trellis coding, $t$--resilient functions, and was motivated by some applications from cryptography \cite{wei}.  It is the study of the generalized Hamming weight of a linear code that motivates our definition of a generalized minimum distance function for any graded ideal in a polynomial ring \cite{helleseth,klove}.

If the rank of $G$ is $s$, then it turns out (see \cite{wei}) that 
\begin{equation}
\label{eq:deltaandhyp}
\delta_r(\mathcal C)=n-{\rm hyp}_r(\mathcal C),
\end{equation}
 where ${\rm hyp}_r(\mathcal C)$, is the maximum number of columns of $G$ that span an $(s-r)$-dimensional vector subspace of $K^s$.  Moreover, if $G$ also has no proportional columns then the columns of $G$ determine the coordinates of $n$ (projective) points in $\mathbb P^{s-1}$, not all contained in a hyperplane. Denote this set  $\mathbb X=\{P_1,\ldots,P_n\}$ and let $I:=I(\mathbb X)\subset S:=K[t_1,\ldots,t_s]$ be the defining ideal of $\mathbb X$. We have:
\begin{itemize}
  \item the (Krull) dimension of $S/I$ is $\dim(S/I)=1$, and the degree is $\deg(S/I)=n$;
  \item the ideal $I$ is given by $I=\mathfrak{p}_1\cap\cdots\cap \mathfrak{p}_n$, where $\mathfrak{p}_i$ is the vanishing ideal of the point $P_i$, so $I$ is unmixed, each associated prime ideal $\mathfrak{p}_i$ is generated by linear forms, and  $I=\sqrt{I}$;
  \item $\displaystyle{\rm hyp}_r(\mathcal C)=\max_{F\in \mathcal F_r}\{\deg(S/(I,F))\}$, where $\mathcal F_r$ is the set of $r$-tuples of linear forms of $S$ that are linearly independent. With this, we can conclude that
      $$\delta_r(\mathcal C)=\deg(S/I)-\max_{F\in \mathcal F_r}\{\deg(S/(I,F))\}.$$
\end{itemize}

A similar approach can be taken for projective Reed--Muller-type codes. Let $\mathbb{X}=\{P_1,\ldots,P_n\}$ be a finite subset of
$\mathbb{P}^{s-1}$. Let $I:=I(\mathbb{X})\subset S=K[t_1,\ldots,t_s]$, be the defining ideal of $\mathbb X$. Via a rescaling of the homogeneous coordinates of the points $P_i$, we can assume that the first non-zero coordinate of each $P_i$ is $1$. 
Fix a degree $d\geq 1$. Because of the assumption on the coordinates of the $P_i$, there is a well-defined $K$-linear map given by the evaluation of the homogeneous polynomials of degree $d$ at each point in $\mathbb X$. This map is given by
\begin{equation*}
{\rm ev}_d\colon S_d\rightarrow K^{n},\ \ \ \ \
f\mapsto
\left(f(P_1),\ldots,f(P_n)\right),
\end{equation*} where $S_d$ denotes the $K$-vector space of homogeneous polynomials of $S$ of degree $d$.  The image of $S_d$ under ${\rm ev}_d$, denoted by  $C_\mathbb{X}(d)$, is
called a {\it projective Reed-Muller-type code\/} of
degree $d$ on $\mathbb{X}$ \cite{duursma-renteria-tapia,Gold,GRT}. The {\it parameters} of the linear
code $C_\mathbb{X}(d)$ are:
\begin{itemize}
\item {\it length\/}: $|\mathbb{X}|=\deg(S/I)$;
\item {\it dimension\/}: $\dim_K C_\mathbb{X}(d)=H_{\mathbb X}(d)$, the Hilbert function of $S/I$ in degree $d$;
\item$r$-th {\it generalized Hamming weight\/}:
$\delta_\mathbb{X}(d,r):=\delta_r(C_\mathbb{X}(d))$.
\end{itemize}

By \cite[Theorem~4.5]{rth-footprint} the $r$-th generalized Hamming weight of a projective Reed--Muller code is given by
$$
\delta_\mathbb{X}(d,r)=\deg(S/I)-\max_{F\in \mathcal F_{d,r}}\{\deg(S/(I,F)\},
$$ where $\mathcal F_{d,r}$ the set of $r$-tuples of forms of degree $d$ in $S$ which are linearly independent over $K$ modulo the ideal $I$ and the maximum is taken to be 0 if $\mathcal F_{d,r}=\emptyset$.

\vskip .1in

As we can see above, the generalized Hamming weights for any linear code can be interpreted using the  language of commutative algebra. Motivated by the notion of generalized Hamming weight described above and following  \cite{rth-footprint} we define {\it generalized minimum distance (GMD) functions} for any homogeneous ideal in a polynomial ring. This allows us to extend  the notion of generalized Hamming weights to codes arising from algebraic schemes, rather than just from reduced sets of points. Another advantage to formulating the notion of generalized minimum distance in the language of commutative algebra is that it allows the use of various homological invariants of graded ideals to study the possible values for these GMD functions.

Let $S=K[t_1,\ldots,t_s]=\oplus_{d=0}^{\infty} S_d$ be a polynomial ring over
a field $K$ with the standard grading and let $I\neq(0)$ be a graded ideal
of $S$.
Given $d,r\in\mathbb{N}_+$, let $\mathcal{F}_{d,r}$ be the set:
$$
\mathcal{F}_{d,r}:=\{\, \{f_1,\ldots,f_r\}\subset S_d\ \mid \overline{f}_1,\ldots,\overline{f}_r\, \mbox{are linearly
independent over }K, (I\colon(f_1,\ldots,f_r))\neq I\},
$$
where $\overline{f}=f+I$ is the class of $f$ modulo $I$, and
$(I\colon(f_1,\ldots,f_r))=\{g\in S \mid gf_i \in I,\, \text{for all $i$}\}$. 
If necessary we denote $\mathcal{F}_{d,r}$ by $\mathcal{F}_{d,r}(I)$.
We denote the {\it degree\/} of $S/I$ by $\deg(S/I)$. 

\begin{definition}
\label{def:GMD}
 Let $I\neq(0)$ be a graded ideal of $S$.
The function
$\delta_I\colon \mathbb{N}_+\times\mathbb{N}_+\rightarrow \mathbb{Z}$ given by
$$
\delta_I(d,r):=\begin{cases} \deg(S/I)-\max\{\deg(S/(I,F)) \mid F\in\mathcal{F}_{d,r}\}&\mbox{if }\mathcal{F}_{d,r}\neq\emptyset,\\
\deg(S/I)&\mbox{if\ }\mathcal{F}_{d,r}=\emptyset,
\end{cases}
$$
is called the {\it generalized minimum distance function\/} of
$I$, or simply the GMD function of $I$. 
\end{definition}
This notion recovers
(Proposition~\ref{jul15-18}) and refines the algebraic-geometric notion of degree.
If $r=1$ one obtains
the minimum distance function of $I$ \cite{hilbert-min-dis}. In this
case we denote $\delta_I(d,1)$ simply by $\delta_I(d)$ and
$\mathcal{F}_{d,r}$ by $\mathcal{F}_{d}$.

The aims of this paper are to study the behavior of
$\delta_I$, to introduce algebraic methods to estimate
this function, and to study the algebraic invariants (minimum
distance function, v-number, regularity, socle degrees) of special ideals that we call Geramita ideals.  Recall that an ideal $I\subset S$ is called {\it unmixed\/} if all its associated primes have the same height; this notion is sometimes
called height unmixed in the literature. We call an ideal $I\subset S$ a \textit{Geramita ideal} if $I$ is an unmixed graded ideal of dimension $1$ whose associated primes are generated by linear forms.
The defining ideal of the scheme of a finite sets of projective fat points and the
unmixed monomial ideals of dimension $1$ are examples of Geramita ideals.

The following function is closely related to $\delta_I$ as illustrated in Eq. \eqref{eq:deltaandhyp}.
\begin{definition}
\label{def:hyp}
Let $I$ be a graded ideal of $S$.
 The function
${\rm hyp}_I\colon \mathbb{N}_+\times\mathbb{N}_+\rightarrow
\mathbb{N}$, given by
$$
{\rm hyp}_I(d,r):=\begin{cases} \max\{\deg(S/(I,F)) \mid
F\in\mathcal{F}_{d,r}\}&\mbox{if }\mathcal{F}_{d,r}\neq\emptyset,\\
0&\mbox{if\ }\mathcal{F}_{d,r}=\emptyset,
\end{cases}$$
is called the {\it hyp function} of $I$. 
\end{definition}
If $r=1$, we denote ${\rm
hyp}_I(d,1)$ by ${\rm hyp}_I(d)$. Finding upper bounds for
${\rm hyp}_I(d,r)$ is equivalent to finding lower bounds for
$\delta_I(d,r)$.
If $I(\mathbb{X})$ is the vanishing ideal
of a finite set $\mathbb{X}$ of reduced projective points, then
${\rm hyp}_{I(\mathbb{X})}(d,1)$ is the maximum number of points of $\mathbb{X}$
contained in a hypersurface of degree $d$ (see
\cite[Remarks~2.7 and 3.4]{tohaneanu-vantuyl}). There is a similar
geometric interpretation for ${\rm hyp}_{I(\mathbb{X})}(d,r)$
\cite[Lemma~3.4]{rth-footprint}.

To compute $\delta_I(d,r)$ is a difficult problem even when $K$ is a finite field and $r=1$.  However, we show that a generalized footprint function, which is more computationally tractable, gives lower bounds for $\delta_I(d,r)$.  Fix a monomial order $\prec$ on $S$. Let ${\rm in}_\prec(I)$ be the initial ideal of $I$ and let $\Delta_\prec(I)$  be the
{\it footprint\/} of $S/I$, consisting of all the {\it standard
monomials\/} of $S/I$ with respect
to $\prec$. The footprint of $S/I$ is also called the {\it
Gr\"obner \'escalier\/} of $I$. Given integers $d,r\geq 1$, let
$\mathcal{M}_{\prec, d, r}$ be the
set of all subsets $M$ of $\Delta_\prec(I)_d:=\Delta_\prec(I)\cap S_d$
with $r$ distinct elements such
that $({\rm in}_\prec(I)\colon(M))\neq {\rm
in}_\prec(I)$.

\begin{definition}
\label{def:fp}
The {\it generalized footprint
function\/} of $I$,
denoted ${\rm fp}_I$, is the function ${\rm fp}_I\colon
\mathbb{N}_+\times\mathbb{N}_+\rightarrow \mathbb{Z}$ given
by
$$
{\rm fp}_I(d,r):=\begin{cases} \deg(S/I)-\max\{\deg(S/({\rm
in}_\prec(I),M))\,\vert\,
M\in\mathcal{M}_{\prec, d,r}\}&\mbox{if }\mathcal{M}_{\prec, d,r}\neq\emptyset,\\
\deg(S/I)&\mbox{if }\mathcal{M}_{\prec, d,r}=\emptyset.
\end{cases}.
$$
\end{definition}

If $r=1$ one obtains the footprint function of $I$ that was
studied in \cite{footprint-ci} from a theoretical point of
view (see \cite{hilbert-min-dis,min-dis-ci} for some applications). In
this case we denote ${\rm fp}_I(d,1)$ simply by ${\rm fp}_I(d)$ and
$\mathcal{M}_{\prec, d,r}$ by $\mathcal{M}_{\prec, d}$. The importance of the footprint function is that it gives a lower bound on the generalized minimum degree function  (Theorem \ref{rth-footprint-lower-bound}) and it is computationally much easier to determine than the generalized minimum degree function. See the Appendix for scripts that implement these computations.

The content of this paper is as follows. In Section~\ref{prelim-section} we
present some of the results and terminology that will be needed
throughout the paper. In some of our results we will assume that there
exists a linear form $h$ that is regular on $S/I$, that is,
$(I\colon h)=I$. There are wide families
of ideals over finite fields
that satisfy this hypothesis, e.g., vanishing ideals of
parameterized codes \cite{algcodes}. Thus our results can be
applied to a variety of Reed--Muller type codes \cite{GRT}, to
monomial ideals, and to ideals that satisfy
$|K|>\deg(S/\sqrt{I})$.

In Section~\ref{mdf-section} we study GMD functions of unmixed graded ideals.
The {\it footprint matrix\/} $({\rm fp}_I(d,r))$
and the {\it weight matrix} $(\delta_I(d,r))$ of $I$ are the
matrices whose $(d,r)$-entries are ${\rm fp}_I(d,r)$ and
$\delta_I(d,r)$, respectively. We show that the entries of each
row of the weight matrix form a non-decreasing sequence and that the entries of each
column of the weight matrix form a non-increasing sequence
(Theorem~\ref{rth-footprint-lower-bound}). We also show that ${\rm fp}_I(d,r)$
is a lower bound for $\delta_I(d,r)$
(Theorem~\ref{rth-footprint-lower-bound}).
This was known when $I$ is the vanishing ideal of a finite set of
projective points
\cite[Theorem~4.9]{rth-footprint}.

Let $I\subset
S$ be an unmixed graded ideal whose
associated primes are generated by linear forms. In
Section~\ref{min-dis-section} we study the minimum distance functions
of these ideals.
For $\delta_I(d)=\delta_I(d,1)$, the \textit{regularity index} of
$\delta_I$, denoted ${\rm reg}(\delta_I)$, is the smallest $d\geq 1$
such that $\delta_I(d)=1$. If $I$ is prime, we set ${\rm
reg}(\delta_I)=1$. The regularity index of $\delta_I$ is the index where the value of this numerical
function stabilizes (Remark~\ref{remark-reg-index}),
named by analogy with the regularity index for
the Hilbert function of a fat point scheme $Z$ which is the
 index where the Hilbert function $H_Z$ of $Z$ stabilizes.

In order to study the behavior of $\delta_I$ we introduce a numerical invariant called the v-{\em number} (Definition \ref{def:v-number}).
We give a description for this invariant in Proposition~\ref{lem:vnumber} that will allow us
to compute it using computer algebra systems, e.g.~{\it Macaulay\/}$2$ \cite{mac2} (Example~\ref{sep5-18-example}).

\noindent {\bf Proposition~\ref{monday-morning}}{\it\ Let
$I\subsetneq\mathfrak{m}\subset S$ be an
unmixed graded ideal whose
associated primes are generated by linear forms. Then
${\rm reg}(\delta_I)={\rm v}(I)$.
}

From the viewpoint of algebraic
coding theory it is important to determine ${\rm reg}(\delta_I)$.
Indeed let $\mathbb{X}$ be a set of projective points over a finite
field $K$, let $C_\mathbb{X}(d)$ be its corresponding Reed-Muller
type code, and let $\delta_\mathbb{X}(d)$ be the minimum distance of
$C_\mathbb{X}(d)$ (see
Section~\ref{Reed-Muller-section}), then $\delta_\mathbb{X}(d)\geq 2$ if
and only if $1\leq d<{\rm reg}(\delta_{I(\mathbb{X})})$. Our results  give an effective method---that can be applied to any Reed-Muller type code---to compute the regularity index of the minimum
distance (Corollary~\ref{min-dis-v-number}, Example~\ref{Hiram-first-counterxample}).

The minimum socle degree ${\rm s}(I)$ of $S/I$
(Definition~\ref{regularity-socle-degree}) was used
in \cite{tohaneanu-vantuyl} to obtain homological lower bounds
for the minimum distance of a fat point scheme $Z$ in
$\mathbb{P}^{s-1}$.
We relate the minimum socle degree, the v-number and the Castelnuovo-Mumford regularity for Geramita ideals in Theorem~\ref{banff-july-22-29-2018}.  For radical ideals it is an open problem whether or
not ${\rm reg}(\delta_I)\leq {\rm reg}(S/I)$
\cite[Conjecture~4.2]{footprint-ci}. In dimension $1$, the
conjecture is true because of Proposition~\ref{monday-morning} and
Theorem~\ref{banff-july-22-29-2018}.  Moreover, via Theorem~\ref{banff-july-22-29-2018}, we can
extend the notion of a Cayley--Bacharach scheme \cite{geramita-cayley-bacharach} by
defining the notion of a Cayley--Bacharach ideal (Definition~\ref{cayley-bacharach-def}). It turns out
that Cayley-Bacharach ideals are connected to Reed--Muller type codes
and to minimum distance functions.

Letting $H_I$ be the Hilbert function of $I$, we have $\delta_I(d)>\deg(S/I)-H_I(d)+1$ for some $d \geq 1$ when $I$ is unmixed of dimension at least $2$ (Proposition~\ref{sep10-18}).  One of our main results is:

\noindent {\bf Theorem~\ref{banff-theorem-1}}{\it\
If  $I\subset S$ is a Geramita ideal and there exists $h\in S_1$ regular on $S/I$, then
$$
\delta_I(d)\leq \deg(S/I)-H_I(d)+1
$$
for $d \geq 1$ or equivalently $H_I(d)-1\leq {\rm hyp}_I(d)$ for $d\geq 1$.
}

This inequality is well known when $I$ is the vanishing ideal of a
finite set of projective points \cite[p.~82]{algcodes}. In this case
the inequality is called the {\it Singleton bound\/}
\cite[Corollary~1.1.65]{tsfasman}. 

Projective Reed--Muller-type codes are  studied in
Section~\ref{Reed-Muller-section}. 

The main result of Section~\ref{Reed-Muller-section} shows that the entries of each column of the weight matrix
$(\delta_{\mathbb{X}}(d,r))$ form a decreasing sequence until they
stabilize.

 In particular one recovers the case when $\mathbb{X}$ is a set,
lying on a projective torus, parameterized by a finite set of monomials
\cite[Theorem~12]{camps-sarabia-sarmiento-vila}. Then we show
that $\delta_\mathbb{X}(d,H_\mathbb{X}(d))$ is equal to $|\mathbb{X}|$
for $d\geq 1$ (Corollary~\ref{jul15-18-coro}).

In Section~\ref{ci-section} we examine minimum distance functions of complete
intersection ideals and show some special cases of the following two
conjectures.

\noindent {\bf Conjecture~\ref{tohaneanu-eisenbud}}{\it\ Let $\mathbb{X}$ be a
finite in $\mathbb{P}^{s-1}$ and suppose that $I=I(\mathbb{X})$ is a
complete intersection generated by $f_1,\ldots,f_{c}$, $c=s-1$, with
$d_i=\deg(f_i)$, and $2\leq d_i\leq d_{i+1}$ for
all $i$.
\begin{itemize}
\item[(a)] $(${\rm Toh\u{a}neanu--Van Tuyl
\cite[Conjecture~4.9]{tohaneanu-vantuyl}}$)$
$\delta_I(1)\geq (d_1-1)d_2\cdots d_{c}$.
\item[(b)] $(${\rm Eisenbud-Green-Harris
\cite[Conjecture~CB10]{Eisenbud-Green-Harris}}$)$
If $f_1,\ldots,f_c$ are quadratic forms, then ${\rm hyp}_I(d)\leq
2^c-2^{c-d}$ for $1\leq d\leq c$ or equivalently $\delta_I(d)\geq
2^{c-d}$ for $1\leq d\leq c$.
\end{itemize}}

We prove part (a) of this conjecture, in a more general setting, when
$I$ is equigenerated, that is, all minimal homogeneous generators have the
same degree (Proposition~\ref{stefan-adam-uniform}, Remark~\ref{stefan-adam-uniform-gen}). The
conjecture also holds for $\mathbb{P}^2$
\cite[Theorem~4.10]{tohaneanu-vantuyl} (Corollary~\ref{stefan-adam-P2}).
According to \cite{Eisenbud-Green-Harris}, part (b) of this conjecture is true for the
following values of $d$: $1, c-1, c$.

For all unexplained
terminology and additional information  we refer to
\cite{CLO,Eisen,Mats} (for the theory of Gr\"obner bases,
commutative algebra, and Hilbert functions), and
\cite{MacWilliams-Sloane,tsfasman} (for the theory of
error-correcting codes and linear codes).

\section{Preliminaries}\label{prelim-section}

In this section we
present some of the results that will be needed throughout the paper
and introduce some more notation. All results of this
section are well-known. To avoid repetitions, we continue to employ
the notations and
definitions used in Section~\ref{intro-section}.

\paragraph{\bf Commutative algebra} Let $I\neq(0)$ be a graded ideal
of $S$ of Krull dimension $k$. The {\it Hilbert function} of $S/I$ is:
$H_I(d):=\dim_K(S_d/I_d)$ for $d=0,1,2,\ldots$,
where $I_d=I\cap S_d$. By a theorem of Hilbert \cite[p.~58]{Sta1},
there is a unique polynomial
$P_I(x)\in\mathbb{Q}[x]$ of
degree $k-1$ such that $H_I(d)=P_I(d)$ for  $d\gg 0$. By convention
the degree of the zero polynomial is $-1$.

The {\it degree\/} or {\it multiplicity\/} of $S/I$ is the
positive integer
$$
\deg(S/I):=\begin{cases} (k-1)!\, \lim_{d\rightarrow\infty}{H_I(d)}/{d^{k-1}}
&\mbox{if }k\geq 1,\\
\dim_K(S/I) &\mbox{if\ }k=0.
\end{cases}
$$
As usual ${\rm ht}(I)$
will denote the height of the ideal $I$. By the {\rm dimension\/}
of $I$ (resp. $S/I$) we mean the Krull dimension of $S/I$ denoted by $\dim(S/I)$.

One of the most useful and well-known facts about the degree is its additivity:

\begin{proposition}{\rm(Additivity of the degree
\cite[Proposition~2.5]{prim-dec-critical})}\label{additivity-of-the-degree}
If $I$ is an ideal of $S$ and
$I=\mathfrak{q}_1\cap\cdots\cap\mathfrak{q}_m$
is an irredundant primary
decomposition, then
$$
\deg(S/I)=\sum_{{\rm ht}(\mathfrak{q}_i)={\rm
ht}(I)}\hspace{-3mm}\deg(S/\mathfrak{q}_i).$$
\end{proposition}

If $F\subset S$, the {\it ideal quotient\/} of $I$ with
respect to $(F)$ is given by $(I\colon(F))=\{h\in S\vert\, hF\subset I\}$.
An element $f$ of $S$ is called a {\it zero-divisor\/} of $S/I$---as an
$S$-module---if there is
$\overline{0}\neq \overline{a}\in S/I$ such that
$f\overline{a}=\overline{0}$, and $f$ is called {\it regular} on
$S/I$ if $f$ is not a zero-divisor. Thus $f$ is a zero-divisor if
and only if $(I\colon f)\neq I$. An associated prime of $I$ is a prime
ideal $\mathfrak{p}$ of $S$ of the form $\mathfrak{p}=(I\colon f)$
for some $f$ in $S$.

\begin{theorem}{\cite[Lemma~2.1.19,
Corollary~2.1.30]{monalg-rev}}\label{zero-divisors} If $I$ is an
ideal of $S$ and
$I=\mathfrak{q}_1\cap\cdots\cap\mathfrak{q}_m$ is
an irredundant primary decomposition with ${\rm
rad}(\mathfrak{q}_i)=\mathfrak{p}_i$, then the set of zero-divisors
$\mathcal{Z}(S/I)$  of $S/I$ is equal to
$\bigcup_{i=1}^m\mathfrak{p}_i$,
and $\mathfrak{p}_1,\ldots,\mathfrak{p}_m$ are the associated primes of
$I$.
\end{theorem}

\begin{definition}
If $I$ is a graded ideal of $S$, the {\it Hilbert series\/} of $S/I$,
denoted $F_I(x)$, is given by
$$F_I(x)=\sum_{d=0}^\infty H_I(d)x^d, \mbox{ where }x\mbox{ is a
variable}. $$
\end{definition}
\begin{theorem}{\rm(Hilbert--Serre \cite[p.~58]{Sta1})}\label{hilbert-serre}
Let $I\subset S$ be a graded ideal of dimension $k$.
Then there is a unique polynomial
$h(x)\in\mathbb{Z}[x]$ such that
$$
F_I(x)=\frac{h(x)}{(1-x)^k}\ \mbox{ and }\ h(1)>0.
$$
\end{theorem}

\begin{remark}\label{mult-vs-h-vector-1}\rm The leading coefficient
of the Hilbert polynomial $P_I(x)$ is equal to
$h(1)/(k-1)!$. Thus $h(1)$ is equal to $\deg(S/I)$.
\end{remark}

\begin{definition} Let $I\subset S$ be a graded ideal.
The $a$-invariant of $S/I$, denoted $a(S/I)$, is the degree of
$F_I(x)$ as a rational
function, that is, $a(S/I)=\deg(h(x))-k$. If
$h(x)=\sum_{i=0}^rh_ix^i$, $h_i\in\mathbb{Z}$, $h_r\neq 0$, the
vector $(h_0,\ldots,h_r)$ is called the $h$-vector of $S/I$.
\end{definition}

\begin{definition}\label{regularity-socle-degree}\rm Let $I\subset S$ be a graded ideal and let
${\mathbf F}$ be the minimal graded free resolution of $S/I$ as an
$S$-module:
\[
{\mathbf F}:\ \ \ 0\rightarrow
\bigoplus_{j}S(-j)^{b_{g,j}}
\stackrel{}{\rightarrow} \cdots
\rightarrow\bigoplus_{j}
S(-j)^{b_{1,j}}\stackrel{}{\rightarrow} S
\rightarrow S/I \rightarrow 0.
\]
The {\it Castelnuovo--Mumford regularity\/} of $S/I$ ({\it
regularity} of $S/I$ for short) and the {\it minimum socle degree}
(s-{\it number} for short) of $S/I$ are defined as
$${\rm reg}(S/I)=\max\{j-i \mid b_{i,j}\neq 0\}\ \mbox{ and }\ {\rm s}(I)={\rm min}\{j-g \mid b_{g,j}\neq 0\}.$$
If $S/I$ is Cohen-Macaulay (i.e. $g=\dim(S)-\dim(S/I)$) and there is a unique $j$ such that $b_{g,j}\neq 0$, then the ring $S/I$ is called {\em level}. In particular, a level ring for which the unique $j$ such that $b_{g,j}\neq 0$ is $b_{g,j}=1$ is called {\em Gorenstein}.
\end{definition}

An excellent reference for the regularity of graded ideals is the book of Eisenbud
\cite{eisenbud-syzygies}.

\begin{definition}\label{definition:index-of-regularity}\rm
The \emph{regularity index} of the Hilbert function of $S/I$, or simply
the \emph{regularity index} of $S/I$, denoted
${\rm ri}(S/I)$, is the least integer $n\geq 0$ such that
$H_I(d)=P_I(d)$ for $d\geq n$.
\end{definition}

The next result is valid over any field; see for instance
\cite[Theorem~5.6.4]{monalg-rev}.

\begin{theorem}\cite{geramita-cayley-bacharach}\label{hilbert-function-dim=1}
Let $I$ be a graded ideal with ${\rm depth}(S/I)>0$. The following
hold.
\begin{itemize}
\item[\rm(i)] If $\dim(S/I)\geq 2$, then $H_I(i)<H_I(i+1)$ for
$i\geq 0$.
\item[\rm (ii)] If $\dim(S/I)=1$, then there is an integer $r$ and a
constant $c$ such that
$$
1=H_I(0)<H_I(1)<\cdots<H_I(r-1)<H_I(i)=c \hspace{.5cm} {\it for}
\ i\geq r.
$$
\end{itemize}
\end{theorem}

\begin{lemma}\label{tohaneanu-aug29-18}
Let $I\subset J\subset S$ be graded ideals of the same
height. The following hold.
\begin{itemize}
\item[(a)] \cite[Lemma~8]{Engheta} If $I$ and $J$ are unmixed, then $I=J$ if and only if $\deg(S/I)=\deg(S/J)$.
\item[(b)] If $I\subsetneq J$, then $\deg(S/I)>\deg(S/J)$.
\end{itemize}
\end{lemma}

\begin{proof} 

(b) Since any associated prime of $J/I$ is an associated prime of $S/I$,
$\dim(J/I)=\dim(S/I)$. From the short exact sequence
$$
0\rightarrow J/I\rightarrow S/I \rightarrow S/J\rightarrow 0
$$
we obtain $\deg(S/I)=\deg(J/I)+\deg(S/J)$. As $J/I$ is not zero,
one has $\deg(S/I)>\deg(S/J)$.
\end{proof}

\begin{lemma}{\cite[p.~122]{monalg-rev}}\label{uniform-regular-sequence} Let $I\subset S$ a
graded ideal of height $r$. If $K$ is infinite and $I$ is minimally
generated by forms of degree $p\geq 1$, then there are
forms $f_1,\ldots,f_m$ of degree $p$ in $I$
such that $f_1,\ldots,f_r$ is a regular sequence and
$I$ is minimally generated by $f_1,\ldots,f_m$.
\end{lemma}

\paragraph{\bf The footprint of an ideal} Let  $\prec$ be a monomial
order on $S$ and let $(0)\neq I\subset S$ be an ideal. If $f$ is a non-zero
polynomial in $S$, the {\it leading
monomial\/} of $f$
is denoted by ${\rm in}_\prec(f)$. The {\it initial ideal\/} of $I$, denoted by
${\rm in}_\prec(I)$, is the monomial ideal given by
${\rm in}_\prec(I)=(\{{\rm in}_\prec(f) \mid f \in I\})$.

 We will use the following
multi-index notation: for $a=(a_1,\ldots,a_s)\in\mathbb{N}^s$, set
$t^a:=t_1^{a_1}\cdots t_s^{a_s}$.
A monomial $t^a$ is called a
{\it standard monomial\/} of $S/I$, with respect
to $\prec$, if $t^a$ is not in the ideal ${\rm in}_\prec(I)$.
A polynomial $f$ is called {\it standard\/} if
$f\neq 0$ and $f$ is a
$K$-linear combination of standard monomials.
The set of standard monomials, denoted $\Delta_\prec(I)$, is called the {\it
footprint\/} of $S/I$. The image of the standard polynomials of
degree $d$, under the canonical map $S\mapsto S/I$,
$x\mapsto \overline{x}$, is equal to $S_d/I_d$, and the
image of $\Delta_\prec(I)$ is a basis of $S/I$ as a $K$-vector space.
This is a classical result of Macaulay (for a modern approach
see \cite[Chapter~5]{CLO}). In
particular, if $I$ is graded, then $H_I(d)$ is the number of standard
monomials of degree $d$.

\begin{lemma}{\cite[p.~3]{carvalho}}\label{nov6-15} Let $I\subset S$ be an ideal generated by
$\mathcal{G}=\{g_1,\ldots,g_r\}$, then
$$
\Delta_\prec(I)\subset\Delta_\prec({\rm in}_\prec(g_1),\ldots,{\rm
in}_\prec(g_r)).
$$
\end{lemma}

\begin{lemma}{\cite[Lemma~4.7]{rth-footprint}}\label{regular-elt-in}
Let $\prec$ be a monomial order, let $I\subset S$ be an ideal, let
$F=\{f_1,\ldots,f_r\}$ be a set of polynomial of $S$ of positive
degree, and let ${\rm in}_\prec(F)=\{{\rm in}_\prec(f_1),\ldots,{\rm
in}_\prec(f_r)\}$ be the set of initial terms of $F$.
If $({\rm in}_\prec(I)\colon({\rm in}_\prec(F)))={\rm in}_\prec(I)$, then
$(I\colon(F))=I$.
\end{lemma}

Let $\prec$ be a monomial order and let
$\mathcal{F}_{\prec,d,r}$ be the set of all subsets
$F=\{f_1,\ldots,f_r\}$ of $S_d$ such that $(I\colon(F))\neq I$,
$f_i$ is a standard polynomial for all $i$,
$\overline{f}_1,\ldots,\overline{f}_r$ are linearly independent over
the field $K$, and ${\rm in}_\prec(f_1),\ldots,{\rm in}_\prec(f_r)$
are distinct monomials.

The next result is useful for computations
with {\it Macaulay\/}$2$ \cite{mac2} (see
Procedure~\ref{procedure-gmdf}).

\begin{proposition}{\cite[Proposition~4.8]{rth-footprint}}\label{march27-17}
The generalized minimum distance function of $I$
is given by the following formula
$$
\delta_I(d,r)=\left\{\begin{array}{ll}\deg(S/I)-\max\{\deg(S/(I,F)) \mid F\in\mathcal{F}_{\prec, d,r}\}&\mbox{if }\mathcal{F}_{\prec,d,r}\neq\emptyset,\\
\deg(S/I)&\mbox{if\ }\mathcal{F}_{\prec,d,r}=\emptyset.
\end{array}\right.
$$
\end{proposition}

An ideal $I\subset S$ is called
{\it radical\/} if $I$ is equal to its radical. The radical of $I$ is
denoted by $\sqrt{I}$. 

\begin{lemma}{\cite[Lemma~3.3]{rth-footprint}}\label{jul11-15} Let
$I\subset S$ be a radical unmixed graded ideal.
If $F=\{f_1,\ldots,f_r\}$ is a set of homogeneous polynomials of
$S\setminus\{0\}$, $(I\colon (F))\neq I$, and $\mathcal{A}$ is
the set of all associated primes
of $S/I$ that contain $F$, then ${\rm ht}(I)={\rm ht}(I,F)$,
$\mathcal{A}\neq\emptyset$, and
$$
\deg(S/(I,F))=\sum_{\mathfrak{p}\in\mathcal{A}}\deg(S/\mathfrak{p}).
$$
\end{lemma}

\section{Generalized minimum distance function of a graded
ideal}\label{mdf-section}

In this section we study the generalized minimum distance function of a graded
ideal. 

Part (c) of the next lemma was known for vanishing ideals
and part (b) for unmixed radical ideals \cite[Proposition~3.5,
Lemma~4.1]{rth-footprint}.

\begin{lemma}\label{degree-initial-footprint} Let $I\subset S$ be an
unmixed graded ideal, let $\prec$ be
a monomial order, and let $F$ be a finite set of homogeneous polynomials of
$S$ such that $(I\colon(F))\neq I$. The following hold.
\begin{itemize}
\item[(a)] ${\rm ht}(I)={\rm ht}(I,F)$
\item[(b)] $\deg(S/(I,F))<\deg(S/I)$ if $I$ is an unmixed ideal and
$(F)\not\subset I$.
\item[(c)] $\deg(S/I)=\deg(S/(I\colon(F)))+\deg(S/(I,F))$ if $I$ is an
unmixed radical ideal.
\item [(d)] \cite[Lemma~4.1]{rth-footprint} $\deg(S/(I,F))\leq\deg(S/({\rm
in}_\prec(I),{\rm in}_\prec(F)))\leq\deg(S/I)$.
\end{itemize}
\end{lemma}

\begin{proof}
(a)  As $I\subsetneq (I\colon (F))$, there is $g\in
S\setminus{I}$ such that $g(F)\subset I$. Hence the ideal
$(F)$ is contained in the set of zero-divisors of $S/I$. Thus, by
Theorem~\ref{zero-divisors} and since $I$ is unmixed, $(F)$ is
contained in an associated
prime ideal $\mathfrak{p}$ of $S/I$ of
height ${\rm ht}(I)$. Thus $I\subset(I,F)\subset\mathfrak{p}$, and
consequently  ${\rm ht}(I)={\rm ht}(I,F)$. Therefore the set of
associated primes of $(I,F)$ of height equal to ${\rm ht}(I)$ is not
empty and is equal to the set of associated primes of $S/I$ that contain $(F)$.  

(b) The inequality follows from part (a) and Lemma \ref{tohaneanu-aug29-18} (b).

(c) Let $\mathfrak{p}_1,\ldots,\mathfrak{p}_m$ be the associated
primes of $S/I$. As $I$ is a radical ideal, one has
the decompositions
$$
I=\bigcap_{i=1}^m\mathfrak{p}_i\ \mbox{ and }\ (I\colon(F))=
\bigcap_{i=1}^m(\mathfrak{p}_i\colon(F)).
$$
\quad Note that $(\mathfrak{p}_i\colon(F))=S$ if
$F\subset\mathfrak{p}_i$ and $(\mathfrak{p}_i\colon
(F))=\mathfrak{p}_i$ if $F\not\subset\mathfrak{p}_i$. Therefore,
using the additivity of the
degree of Proposition~\ref{additivity-of-the-degree} and Lemma~\ref{jul11-15}, we get
$$
\deg(S/(I\colon(F)))=\sum_{F\not\subset\mathfrak{p}_i}\deg(S/\mathfrak{p}_i)
\ \mbox{ and }\ \deg(S/(I,F))=\sum_{F\subset\mathfrak{p}_i}\deg(S/\mathfrak{p}_i).
$$
\quad Thus
$\deg(S/I)=\sum_{i=1}^m\deg(S/\mathfrak{p}_i)=
\deg(S/(I\colon(F)))+\deg(S/(I,F))$.

\end{proof}

\begin{definition} Let $I\subset S$ be a graded ideal. A sequence
$f_1,\ldots,f_r$ of elements of $S$ is called a $(d,r)$-{\it
sequence\/} of $S/I$
if the set $F=\{f_1,\ldots,f_r\}$ is in $\mathcal{F}_{d,r}$
\end{definition}

\begin{lemma} Let $I\subset S$ be a graded ideal. A sequence
$f_1,\ldots,f_r$ is a $(d,r)$-sequence of $S/I$ if and only if
the following conditions hold
\begin{itemize}
\item[(a)] $f_1,\ldots,f_r$ are homogeneous polynomials of $S$ of
degree $d\geq 1$,
\item[(b)] $(I\colon(f_1,\ldots,f_r))\neq I$, and
\item[(c)] $f_i\notin(I,f_1,\ldots,f_{i-1})$ for $i=1,\ldots,r$, where
we set $f_0=0$.
\end{itemize}
\end{lemma}

\begin{proof} The proof is straightforward.
\end{proof}

\begin{definition}\label{vasconcelos-function} If $I\subset S$ is a
graded ideal, the {\it Vasconcelos function\/} of $I$ is the
function
$\vartheta_I\colon \mathbb{N}_+\times\mathbb{N}_+\rightarrow \mathbb{N}$ given by
$$
\vartheta_I(d,r):=\left\{\begin{array}{ll}\min\{\deg(S/(I\colon(F)))\vert\,
F\in\mathcal{F}_{d,r}\}&\mbox{if }\mathcal{F}_{d,r}\neq\emptyset,\\
\deg(S/I)&\mbox{if\ }\mathcal{F}_{d,r}=\emptyset.
\end{array}\right.
$$
\end{definition}

The next result was shown in \cite[Theorem 4.5]{rth-footprint} for
vanishing ideals over finite fields.

\begin{theorem}\label{rth-min-dis-vi}
Let $I\subset S$ be a graded unmixed radical ideal.
Then
$$\vartheta_I(d,r)=\delta_I(d,r)\
\mbox{ for }d\geq 1\mbox{ and }1\leq r\leq H_I(d).$$
\end{theorem}

\demo If $\mathcal{F}_{d,r}=\emptyset$, then
$\delta_I(d,r)$ and $\vartheta_I(d,r)$ are equal to
$\deg(S/I)$. Now assume that
$\mathcal{F}_{d,r}\neq \emptyset$.
Using Lemma~\ref{degree-initial-footprint}(c), we obtain
\begin{eqnarray*}
\vartheta_{I}(d,r)&=&\min\{\deg(S/(I\colon(F)))\vert\,
F\in \mathcal{F}_{d,r}\}\\
 &=&\min\{\deg(S/I)-\deg(S/(I,F))\vert\,
F\in \mathcal{F}_{d,r}\} \\
&=&\deg(S/I)-\max\{\deg(S/(I,F))\vert\,
F\in \mathcal{F}_{d,r}\}=\delta_{I}(d,r).\quad \Box
\end{eqnarray*}

As the next result shows for $r=1$ we do not need the assumption
that $I$ is a radical ideal. For $r\geq 2$ this assumption is
essential, as shown in the next Example~\ref{aug12-18}.

\begin{example}\label{aug12-18} Let $I$  be the ideal $(t_1^2,t_1t_2,t_2^2)$ of the
polynomial ring $S=K[t_1,t_2]$ over a field $K$ and let
$F=\{t_1,t_2\}$. Then $(I\colon (F))=(I,F)=(t_1,t_2)$ and
$$
3=\deg(S/I)\neq\deg(S/(I\colon (F)))+\deg(S/(I,F))=2.
$$
\end{example}

\begin{theorem}{\rm\cite[Theorem~4.4]{hilbert-min-dis}}\label{wolmer-obs}
Let $I\subset S$ be an unmixed graded ideal.
If $\mathfrak{m}=(t_1,\ldots,t_s)$ and $d\geq 1$ is an integer such
that $\mathfrak{m}^d\not\subset I$, then
\begin{equation*}
\delta_I(d)=\min\{\deg(S/(I\colon f))\,\vert\, f\in S_d\setminus
I\}.
\end{equation*}
\end{theorem}

Recall from the introduction that the definition of $\delta_I(d,r)$ was motivated by the notion of
generalized Hamming weight of a linear code \cite{helleseth,wei}. The following compilation of facts reflects the monotonicity of the generalized minimum distance function with respect of its two input values for the case of linear codes corresponding to reduced sets of points.

\begin{theorem}\label{jul13-18} Let $C$ be a linear code of length $m$ and dimension
$k$. The following hold.
\begin{itemize}
\item[(a)] \cite[Theorem~1, Corollary~1]{wei}
$1\leq\delta_1(C)<\cdots<\delta_k(C)\leq m$.
\item[(b)] \cite[Corollary~1.1.65]{tsfasman} $r\leq \delta_r(C)\leq
m-k+r$ for $r=1,\ldots,k$.
\item[(c)] If $\delta_1(C)=m-k+1$, then $\delta_r(C)=m-k+r$ for
$r=1,\ldots,k$.
\end{itemize}
\end{theorem}
\begin{proof} (c):  By (a), one has $m-k+1=\delta_1(C)\leq
\delta_i(C)-(r-1)$. Thus $m-k+r\leq \delta_i(C)$ and, by (b),
equality holds.
\end{proof}

Below we consider more generally the behavior of the generalized minimum distance function and the footprint function for arbitrary graded ideals. The next result shows that the entries of any row (resp. column)
of the weight matrix of $I$ form a non-decreasing (resp.
non-increasing) sequence. Parts (a)-(c) of the next result are broad generalizations of
\cite[Theorem~4.9]{rth-footprint} and
\cite[Theorem~3.6]{footprint-ci}.

\begin{theorem}\label{rth-footprint-lower-bound}
Let $I\subset S$ be an unmixed graded ideal, let $\prec$ be a monomial
order on $S$, and let $d\geq 1$, $r\geq 1$ be
integers. The following hold.
\begin{itemize}
\item[\rm(a)] ${\rm fp}_I(d,r)\leq \delta_I(d,r)$ for $1\leq r\leq H_I(d)$.
\item[\rm(b)] $\delta_I(d,r)\geq 1$.
\item[\rm(c)] ${\rm fp}_I(d,r)\geq 1$ if ${\rm in}_\prec(I)$ is unmixed.
\item[\rm(d)] $\delta_I(d,r)\leq \delta_I(d,r+1)$.
\item[\rm(e)] If there is $h\in S_1$ regular on $S/I$, then
$\delta_I(d,r)\geq \delta_I(d+1,r)\geq 1$.
\end{itemize}
\end{theorem}

\demo (a) If $\mathcal{F}_{d,r}=\emptyset$, then
$\delta_I(d,r)=\deg(S/I)\geq {\rm fp}_I(d,r)$. Now assume
$\mathcal{F}_{d,r}\neq\emptyset$. Let $F$ be any set in
$\mathcal{F}_{\prec,d,r}$. By Lemma~\ref{regular-elt-in}, ${\rm
in}_\prec(F)$ is in $\mathcal{M}_{\prec,d,r}$, and by
Lemma~\ref{degree-initial-footprint},
$\deg(S/(I,F))\leq \deg(S/({\rm in}_\prec(I),{\rm in}_\prec(F)))$.
Hence, by Proposition~\ref{march27-17} and
Lemma~\ref{degree-initial-footprint}(b),
${\rm fp}_I(d,r)\leq
\delta_I(d,r)$.

(b) If $\mathcal{F}_{d,r}=\emptyset$, then $\delta_I(d,r)=\deg(S/I)\geq 1$,
and if $\mathcal{F}_{d,r}\neq\emptyset$, then using
Lemma~\ref{degree-initial-footprint}(b) it follows
that $\delta_I(d,r)\geq 1$.

(c) If $\mathcal{M}_{\prec,d,r}=\emptyset$, then
${\rm fp}_I(d,r)=\deg(S/I)\geq 1$. Next assume that
$\mathcal{M}_{\prec,d,r}$ is not empty and pick
$M$ in $\mathcal{M}_{\prec, d,r}$ such that
$$
{\rm fp}_I(d,r)=\deg(S/I)-\deg(S/({\rm in}_\prec(I),M)).
$$
\quad As ${\rm in}_\prec(I)$ is unmixed, by
Lemma~\ref{degree-initial-footprint}(b), ${\rm fp}_I(d,r)\geq 1$.

(d) If  $\mathcal{F}_{d,r+1}$ is empty, then
$ \delta_I(d,r)\leq\deg(S/I)=\delta_I(d,r+1)$. We may then assume
$\mathcal{F}_{d,r+1}$ is not empty and pick $F=\{f_1,\ldots,f_{r+1}\}$
in $\mathcal{F}_{d,r+1}$ such that ${\rm
hyp}_I(d,r+1)=\deg(S/(I,F))$. Setting $F'=\{f_1,\ldots,f_r\}$ and
noticing
that $I\subsetneq (I\colon(F))\subset(I\colon(F'))$, we get
$F'\in\mathcal{F}_{d,r}$. By the proof of Lemma~\ref{degree-initial-footprint}, one has
${\rm ht}(I)={\rm ht}(I,F)={\rm ht}(I,F')$. Taking
Hilbert
functions in the exact sequence
$$
0\longrightarrow (I,F)/(I,F')\longrightarrow S/(I,F')\longrightarrow
S/(I,F)\longrightarrow 0
$$
it follows that $\deg(S/(I,F'))\geq \deg(S/(I,F))$. Therefore
$$
{\rm hyp}_I(d,r)\geq \deg(S/(I,F'))\geq \deg(S/(I,F))={\rm
hyp}_I(d,r+1)\Rightarrow \delta_I(d,r)\leq \delta_I(d,r+1).
$$

(e) By part (b), $\delta_I(d,r)\geq 1$ for $d\geq 1$. Assume
$\mathcal{F}_{d,r}=\emptyset$. Then $\delta_I(d,r)=\deg(S/I)$.
If the set $\mathcal{F}_{d+1,r}$ is empty, one has
$$
\delta_I(d,r)=\delta_I(d+1,r)=\deg(S/I).
$$
\quad If the set $\mathcal{F}_{d+1,r}$ is not empty, there is
$F\in\mathcal{F}_{d+1,r}$ such that
$$
\delta_I(d+1,r)=\deg(S/I)-\deg(S/(I,F))\leq \deg(S/I)=\delta_I(d,r).
$$
\quad Thus we may now assume $\mathcal{F}_{d,r}\neq\emptyset$. Pick
$F=\{f_1,\ldots,f_r\}$ in $\mathcal{F}_{d,r}$ such that
$$
\delta_I(d,r)=\deg(S/I)-\deg(S/(I,F)).
$$
\quad
By assumption there exists $h\in S_1$ such that $(I\colon h)=I$.
Hence the set
$h\overline{F}=\{h\overline{f}_i\}_{i=1}^r$ is linearly
independent over $K$, $hF\subset S_{d+1}$, and
$$
I\subsetneq (I\colon F)\subset(I\colon hF),
$$
that is, $hF$ is in $\mathcal{F}_{d+1,r}$. Note that there exists
$\mathfrak{p}\in{\rm Ass}(S/I)$ that contains $(I,F)$ (see
Lemma~\ref{degree-initial-footprint}(a)). Hence the ideals $(I,F)$ and
$(I,hF)$ have the same height because a prime ideal
$\mathfrak{p}\in{\rm Ass}(S/I)$ contains $(I,F)$ if and only if
$\mathfrak{p}$ contains $(I,hF)$. Therefore taking Hilbert
functions in the exact sequence
$$
0\longrightarrow (I,F)/(I,hF)\longrightarrow S/(I,hF)\longrightarrow
S/(I,F)\longrightarrow 0
$$
it follows that $\deg(S/(I,hF))\geq \deg(S/(I,F))$. As a consequence
we get
\begin{eqnarray*}
\delta_I(d,r)&=&\deg(S/I)-\deg(S/(I,F))\geq \deg(S/I)-\deg(S/(I,hF))\\
&\geq&\deg(S/I)-\max\{\deg(S/(I,F'))\vert\,
F'\in\mathcal{F}_{d+1,r}\}=\delta_I(d+1,r).\ \Box
\end{eqnarray*}

\begin{remark}\label{remark-reg-index}
(a)
Let $I$ be a non-prime ideal and let $\mathfrak{p}$ be an associated prime of $I$.
There is $f\in S_d$, $d\geq 1$, such that $(I\colon f)=\mathfrak{p}$. Note that
$f\in \mathcal{F}_d$. By Theorem~\ref{wolmer-obs} one has
$\delta_I(d)=1$.

(b) If $\dim(S/I)\geq 1$, then ${\rm reg}(\delta_I)$ is the
smallest $n\geq 1$ such that $\delta_I(d)=1$ for $d\geq n$. This
follows from Theorems~\ref{wolmer-obs} and
\ref{rth-footprint-lower-bound}.
\end{remark}

\begin{example}\label{seceleanu-example-determinantal}
Let $S=K[t_1,\ldots,t_6]$ be
a polynomial ring over the finite field $K=\mathbb{F}_3$
and let $I$ be the ideal $(t_1t_6-t_3t_4,t_2t_6-t_3t_5)$.
The regularity and the degree of $S/I$ are $2$ and $4$, respectively,
and $H_I(1)=6$, $H_I(2)=19$. Using
Procedure~\ref{procedure-gmdf} and Theorem~\ref{rth-footprint-lower-bound}(a) we obtain:
$$
({\rm fp}_I(d,r))=\left[\begin{array}{ccccccccccccc}
1&3 & 4 & 4 & 4 & 4 &\infty\\
1& 1 & 1 & 1 & 2 &3&3
\end{array}\right],\ d=1,2 \mbox{ and }r=1,\ldots,7,
$$
and $(\delta_I(1,1),\ldots,\delta_I(1,5))=(3,3,4,4,4)$.
\end{example}

\begin{definition}\rm
If ${\rm fp}_I(d)=\delta_I(d)$ for $d\geq 1$, we say that $I$ is a
{\it Geil--Carvalho ideal\/}. If ${\rm fp}_I(d,r)=\delta_I(d,r)$ for
$d\geq 1$ and $r\geq 1$, we say that $I$ is a
{\it strongly Geil--Carvalho ideal\/}.
\end{definition}

The next result generalizes \cite[Proposition~3.11]{min-dis-ci}.

\begin{proposition}\label{geil-carvalho-monomial} If $I$ is an unmixed monomial ideal and $\prec$ is
any monomial order, then $\delta_I(d,r)={\rm fp}_I(d,r)$ for $d\geq 1$
and $r\geq 1$, that is, $I$ is a strongly Geil--Carvalho ideal.
\end{proposition}

\begin{proof} The inequality $\delta_I(d,r)\geq {\rm fp}_I(d,r)$ follows
from Theorem~\ref{rth-footprint-lower-bound}(a). To show the reverse inequality
notice that $\mathcal{M}_{\prec, d,r}\subset
\mathcal{F}_{\prec, d,r}$ because one has $I={\rm in}_\prec(I)$. Also
notice that $\mathcal{M}_{\prec, d,r}=\emptyset$ if and only if
$\mathcal{F}_{\prec, d,r}=\emptyset$, this follows from
the proof of \cite[Proposition~4.8]{rth-footprint}. Therefore one has
${\rm fp}_I(d,r)\geq \delta_I(d,r)$.
\end{proof}

\begin{proposition}\label{jul15-18} If $I\subset S$ is an unmixed graded
ideal and $\dim(S/I)\geq 1$, then
$$
\delta_I(d,H_I(d))=\deg(S/I)\ \mbox{ for }\
d\geq 1.
$$
\end{proposition}

\begin{proof} We set $r=H_I(d)$. It suffices to show that
$\mathcal{F}_{d,r}=\emptyset$. We proceed by contradiction. Assume
that $\mathcal{F}_{d,r}$ is not empty and let $F=\{f_1,\ldots,f_r\}$
be an element of $\mathcal{F}_{d,r}$. Let
$\mathfrak{p}_1,\ldots,\mathfrak{p}_m$ be
the associated primes
of $I$. As $I\subsetneq(I\colon(F))$, we can pick $g\in S$ such
that $g(F)\subset
I$ and $g\notin I$. Then $(F)$ is contained
$\cup_{i=1}^m\mathfrak{p}_i$, and consequently
$(F)\subset\mathfrak{p}_i$ for some $i$.  Since $r=H_I(d)$, one has
$$
S_d/I_d=K\overline{f}_1\oplus\cdots\oplus K\overline{f}_r\ \Rightarrow\
S_d=Kf_1+\cdots+Kf_r+I_d.
$$
\quad Hence $S_d\subset \mathfrak{p}_i$, that is,
$\mathfrak{m}^d\subset\mathfrak{p}_i$, where
$\mathfrak{m}=(t_1,\ldots,t_s)$. Therefore
$\mathfrak{p}_i=\mathfrak{m}$, a contradiction because $I$ is unmixed
and $\dim(S/I)\geq 1$.
\end{proof}

\begin{example}\label{seceleanu-example-ci}
Let $S=K[t_1,t_2,t_3]$ be
a polynomial ring over a field $K$ and let $({\rm
fp}_{I}(d,r))$ and
$(\delta_I(d,r))$ be the footprint
matrix and the weight matrix of the ideal $I=(t_1^3,t_2t_3)$.
The regularity and the
degree of $S/I$ are $3$ and $6$. Using
Procedure~\ref{procedure-footprint-matrix} we obtain:
$$
({\rm fp}_I(d,r))=\left[\begin{array}{ccccccccccccc}
3&5 & 6 & \infty & \infty & \infty \\
2& 3 & 4 & 5 & 6 &\infty \\
1& 2 & 3 & 4 & 5 & 6
\end{array}\right].
$$

\quad If $r>H_I(d)$, then $\mathcal{M}_{\prec,
d,r}=\emptyset$ and the  $(d,r)$-entry of this matrix is
equal to $6$, but in this case we write $\infty$ for
computational reasons. Therefore, by
Proposition~\ref{geil-carvalho-monomial}, $({\rm
fp}_I(d,r))$ is equal to $(\delta_I(d,r))$. Setting
$F=\{t_1^2t_2,t_1t_2^2,t_1t_3^2,t_1^2t_3\}$ and
$F'=\{t_1^2t_2,t_1t_2^2,t_1t_3^2+t_2^3,t_1^2t_3\}$ ,
we get
$$
\delta_I(3,4)=\deg(S/I)-\deg(S/(I,F)=4\ \mbox{ and }\
\deg(S/I)-\deg(S/(I,F')=5.
$$
\quad Thus $\delta_I(3,4)$ is attained at $F$.
\end{example}

\section{Minimum distance function of a graded
ideal}\label{min-dis-section}

In this section we study minimum distance functions of unmixed graded
ideals whose associated primes are generated by linear
forms and the algebraic invariants of Geramita ideals.

\subsection{Minimum distance function for unmixed ideals}
We begin by introducing the following numerical invariant which will be used to express the regularity index of the minimum distance function (Proposition \ref{monday-morning}).
\begin{definition}
\label{def:v-number}
The v-\textit{number} of a graded ideal $I$,
denoted ${\rm v}(I)$,
is given by
$$
{\rm v}(I):=\begin{cases}\min\{d\geq 1 \mid \,\, \text{there exists $f \in S_d$ and $\mathfrak{p} \in {\rm Ass}(I)$ with $(I\colon f) =\mathfrak{p}$} \} & \mbox{ if }\ I\subsetneq\mathfrak{m},\\
0 &\mbox{ if }\ I=\mathfrak{m},
\end{cases}$$
where ${\rm Ass}(I)$ is the set of associated primes of $S/I$ and
$\mathfrak{m}=(t_1,\ldots,t_s)$ is the irrelevant maximal ideal of $S$.
\end{definition}
The $v$-number is finite for any graded ideal by the definition of associated primes. If $\mathfrak{p}$ is a prime ideal and $\mathfrak{p}\neq \mathfrak{m}$, then ${\rm v}(\mathfrak{p})=1$.

Let $I\subsetneq\mathfrak{m}\subset S$ be a graded ideal and let
$\mathfrak{p}_1,\ldots,\mathfrak{p}_m$ be its associated primes. One
can define the v-number of $I$ locally at each $\mathfrak{p}_i$ by
$$
{\rm v}_{\mathfrak{p}_i}(I):=\mbox{min}\{d\geq 1\, \vert\,
\exists f\in S_d \mbox{ with }(I\colon f)=\mathfrak{p}_i\}.
$$
\quad The {\rm v}-number of $I$ is equal to ${\rm min}\{{\rm
v}_{\mathfrak{p}_1}(I),\ldots,{\rm v}_{\mathfrak{p}_m}(I)\}$. If
$I=I(\mathbb{X})$ is the vanishing ideal
of a finite set $\mathbb{X}=\{P_1,\ldots,P_m\}$ of reduced projective points and
$\mathfrak{p}_i$ is the vanishing ideal of $P_i$, then ${\rm
v}_{\mathfrak{p}_i}(I)$ is the degree of $P_i$ in $\mathbb{X}$ in the
sense of \cite[Definition~2.1]{geramita-cayley-bacharach}.

We give an alternate description for the v-number using initial
degrees of certain modules. This will allow us to compute the v-number
using {\it
Macaulay\/}$2$ \cite{mac2} (see Example~\ref{sep5-18-example}).
For a graded module $M\neq 0$ we denote
$\alpha(M)=\min\{\deg(f) \mid f\in M, f\neq 0\}$. By convention, for $M=0$ we set $\alpha(0)=0$.

\begin{proposition}
\label{lem:vnumber}
 Let $I\subset S$ be an unmixed graded ideal. Then
$I\subsetneq(I\colon\mathfrak{p})$ for $\mathfrak{p}\in{\rm Ass}(I)$,
 $$
{\rm v}(I)=\min\{\alpha\left((I\colon\mathfrak{p})/{I}\right)\vert\,
\mathfrak{p}\in{\rm Ass}(I)\},
$$
and $\alpha\left((I\colon\mathfrak{p})/{I}\right)={\rm
v}_{\mathfrak{p}}(I)$ for
$\mathfrak{p}\in{\rm Ass}(I)$.
\end{proposition}
\begin{proof} The strict inclusion $I\subsetneq(I\colon\mathfrak{p})$
follows from the equivalence of Eq.~(\ref{oct8-18}) below. As a preliminary step of the proof of the equality we establish that for a
prime $\mathfrak{p}\in{\rm Ass}(I)$ we have
\begin{equation}\label{oct8-18}
(I\colon f)=\mathfrak{p}  \text{ if and only if }
f\in(I\colon \mathfrak{p})\setminus I.
\end{equation}
\quad If $(I\colon f)=\mathfrak{p}$, it is clear that we have
$f\in(I\colon \mathfrak{p})$ and since $(I\colon f)\neq S$ it follows that $f\notin I$.
Conversely, if $f\in(I\colon \mathfrak{p})\setminus I$,  then
$\mathfrak{p}\subset(I\colon f)$. Let $\mathfrak{q}\in{\rm
Ass}(I\colon f)$,
which is a nonempty set since $f\notin I$.  Since ${\rm Ass}(I\colon
f)\subset{\rm Ass}(I)$
and $I$ is height unmixed, we have
$\rm{ht}(\mathfrak{q})=\rm{ht}(\mathfrak{p})$
and $\mathfrak{p}\subset(I\colon f)\subset \mathfrak{q}$.
It follows that $\mathfrak{p}=(I\colon f)=\mathfrak{q}$.

The equivalence of Eq.~(\ref{oct8-18}) implies that
$\alpha\left((I\colon\mathfrak{p})/{I}\right)={\rm
v}_{\mathfrak{p}}(I)$, and shows the
equality
$$\{f \mid (I\colon f)=\mathfrak{p} \text{ for some
}\mathfrak{p}\in{\rm Ass}(I)\}=
\bigcup_{\mathfrak{p}\in{\rm Ass}(I)}(I\colon
\mathfrak{p})\setminus{I}.$$
\quad The claim now follows by
considering the minimum degree of a homogeneous element in the above sets.
\end{proof}

\begin{example}\label{sep5-18-example}
Let $S=\mathbb{Q}[t_1,t_2,t_3,t_4]$ be a polynomial ring over the
rational numbers and let $I$ be the ideal of $S$ given by
$$
I=(t_2^{10},t_3^9,t_4^4,t_2t_3t_4^3)\cap(t_1^4,t_3^4,t_4^3,t_1t_3t_4^2)
\cap(t_1^4,t_2^5,t_4^3)\cap(t_1^3,t_2^5,t_3^{10}).
$$
\quad The associated primes of $I$ are $\mathfrak{p}_1=(t_2,t_3,t_4),
\mathfrak{p}_2=(t_1,t_3,t_4), \mathfrak{p}_3=(t_1,t_2,t_4),
\mathfrak{p}_4=(t_1,t_2,t_3)$. Using Proposition~\ref{lem:vnumber}
together with Procedure~\ref{sep5-18} we get ${\rm s}(I)=10$,
${\rm v}(I)=12$, ${\rm reg}(S/I)=19$, ${\rm
v}_{\mathfrak{p}_1}(I)=12$,
${\rm v}_{\mathfrak{p}_2}(I)=15$, ${\rm v}_{\mathfrak{p}_i}(I)=18$
for $i=3,4$. Thus the minimum socle degree ${\rm s}(I)$ can be
smaller than the ${\rm v}$-number ${\rm v}(I)$.
\end{example}

\begin{corollary}\label{socle-vnumber-dim=0}
If $I\subsetneq\mathfrak{m}$ is a graded ideal of
$S$ and $\dim(S/I)=0$, then the minimum socle degree ${\rm s}(I):=\alpha((I:\frak m)/I)$ of $S/I$ is equal
to ${\rm v}(I)$.
\end{corollary}

\begin{proof} The socle of $S/I$ is given by ${\rm
Soc}(S/I)=(I\colon\mathfrak{m})/I$. Thus, by
Proposition~\ref{lem:vnumber}, one has the equality
${\rm s}(I)={\rm v}(I)$.
\end{proof}

This corollary does not hold in dimension $1$. There are examples of
Geramita monomial ideals satisfying the strict inequality ${\rm s}(I)<{\rm
v}(I)$ (see Example~\ref{sep5-18-example}). If $S/I$ is
a Cohen--Macaulay ring, the
socle is understood to be the socle of some
Artinian reduction of $S/I$ by linear forms.

\begin{example}\label{Hiram-first-counterxample}
Let $K$ be the finite field $\mathbb{F}_3$ and let $\mathbb{X}$ be the
following set of points in $\mathbb{P}^2$:
$$
\begin{matrix}
[(1,0,1)],&[(1,0,0)],&[(1,0,2)],&[(1,1,0)],&[(1,1,1)],\cr
[(1,1,2)],&[(0,0,1)],&[(0,1,0)],&[(0,1,1)],&[(0,1,2)].
\end{matrix}
$$
\quad Using Propositions~\ref{lem:vnumber} and \ref{monday-morning},
together with Procedure~\ref{sep12-18}, we get ${\rm v}(I)={\rm
reg}(\delta_\mathbb{X})=3$, ${\rm
reg}(S/I)=4$, $\delta_\mathbb{X}(1)=6$, $\delta_\mathbb{X}(2)=3$, and
$\delta_\mathbb{X}(d)=1$ for $d\geq 3$. The vanishing
ideal of $\mathbb{X}$ is generated by $t_1t_2^2-t_1^2t_2$,
$t_1t_3^3-t_1^3t_3$, and $t_2t_3^3-t_2^3t_3$.
\end{example}

\begin{proposition}\label{monday-morning}
Let
$I\subsetneq\mathfrak{m}\subset S$ be an
unmixed graded ideal whose
associated primes are generated by linear forms. Then
${\rm reg}(\delta_I)={\rm v}(I)$.
\end{proposition}

\begin{proof} Let $\mathfrak{p}_1,\ldots,\mathfrak{p}_m$ be the
associated primes of $I$. We may assume that $I$ not a prime ideal, otherwise
${\rm reg}(\delta_I)={\rm v}(I)=1$. If $d_1={\rm v}(I)$, there are $f\in S_{d_1}$
and $\mathfrak{p}_i$ such that $(I\colon f)=\mathfrak{p}_i$. Then, by
Theorem~\ref{wolmer-obs}, one has $\delta_I(d_1)=1$. Thus
${\rm reg}(\delta_I)\leq{\rm v}(I)$.

To show the reverse inequality set we set $d_0={\rm reg}(\delta_I)$.
Then $\delta_I(d_0)=1$. Note that $\mathfrak{m}^{d_0}\not\subset I$;
otherwise $\mathcal{F}_{d_0}(I)=\emptyset$ and by definition
$\delta_I(d_0)$ is equal to $\deg(S/I)$, a contradiction because
$I\subsetneq\mathfrak{m}$ and by Lemma~\ref{tohaneanu-aug29-18}
$\deg(S/I)>1$. Then, by Theorem~\ref{wolmer-obs}, there
is $f\in S_{d_0}\setminus I$ such
that $\delta_I(d_0)=\deg(S/(I\colon f))=1$. Let
$I=\cap_{i=1}^m\mathfrak{q}_i$ be the minimal primary
decomposition of $I$, where $\mathfrak{q}_i$ is a
$\mathfrak{p}_i$-primary ideal. Note that $(\mathfrak{q}_i\colon f)$
is a primary ideal if $f\notin \mathfrak{q}_i$ because
$S/(\mathfrak{q}_i\colon f)$ is embedded in $S/\mathfrak{q}_i$. Thus
the primary decomposition of $(I\colon f)$ is $\cap_{f\notin
\mathfrak{q}_i}(\mathfrak{q}_i\colon f)$. Therefore, by the additivity of the
degree of Proposition~\ref{additivity-of-the-degree}, we get that
$(I\colon f)=(\mathfrak{q}_k\colon f)$ for some $k$ such that $f\notin
\mathfrak{q}_k$ and $\deg(S/(\mathfrak{q}_k\colon f))=1$.
Since $S/\mathfrak{p}_k$ has also degree $1$ and
$(\mathfrak{q}_k\colon f)\subset\mathfrak{p}_k$, by
Lemma~\ref{tohaneanu-aug29-18}, we get $(I\colon f)=(\mathfrak{q}_k\colon f)=\mathfrak{p}_k$, and
consequently ${\rm v}(I)\leq{\rm reg}(\delta_I)$.
\end{proof}

\begin{corollary}\label{md-decreasing}
Let $I\subset S$ be an unmixed radical graded ideal. If all the associated primes
of $I$ are generated by linear forms and $v={\rm v}(I)$ is its
${\rm v}$-number, then
$$
\delta_I(1)>\cdots>\delta_I(v-1)>\delta_I(v)=\delta_I(d)=1\
\mbox{ for }\ d\geq v.
$$
\end{corollary}

\begin{proof} It follows from \cite[Theorem~3.8]{footprint-ci} and
Proposition~\ref{monday-morning}.
\end{proof}

The minimum distance function behaves well asymptotically.

\begin{corollary} Let $I\subsetneq\mathfrak{m}\subset S$ be an
unmixed graded ideal of dimension $\geq 1$ whose
associated primes are generated by linear forms. Then
$\delta_I(d)=1$ for $d\geq {\rm v}(I)$.
\end{corollary}

\begin{proof} This follows from Remark~\ref{remark-reg-index}(b) and
Proposition~\ref{monday-morning}.
\end{proof}

The next result relates the minimum socle degree and the v-number.

\begin{proposition}\label{s-v-numbers}
Let $I\subset S$ be an unmixed non-prime graded ideal whose associated primes are
generated by linear forms and let $h\in S_1$ be a regular element on
$S/I$. The following hold:
\begin{itemize}
\item[(a)] If $\delta_I(d)=\deg(S/(I\colon f))$,
$f\in\mathcal{F}_d\cap(I,h)$, then $d\geq 2$ and $\delta_I(d)=\delta_I(d-1)$.
\item[(b)] If $S/I$ is Cohen--Macaulay, then
${\rm v}(I,h)\leq{\rm v}(I)$.
\item[(c)] If $K$ is infinite and $S/I$ is Cohen--Macaulay, then ${\rm s}(I)\leq {\rm
v}(I)$.
\end{itemize}
\end{proposition}

\begin{proof} (a) Writing $f=g+f_1h$, for some $g\in I_d$ and $f_1\in
S_{d-1}$, one has $(I\colon f)=(I\colon f_1)$. Note that $d\geq 2$,
otherwise if $d=1$, then $(I\colon f)=I$, a contradiction because
$f\in\mathcal{F}_d$. Therefore noticing that
$f_1\in\mathcal{F}_{d-1}$,
by Theorems~\ref{wolmer-obs} and
\ref{rth-footprint-lower-bound}, we obtain
$$
\delta_I(d)=\deg(S/(I\colon f))=\deg(S/(I\colon f_1))\geq
\delta_I(d-1)\geq \delta_I(d)\ \Rightarrow\ \delta_I(d)=\delta_I(d-1).
$$
\quad (b) We set $v={\rm v}(I)$. By Proposition~\ref{lem:vnumber}
there is an associated
prime $\mathfrak{p}$ of $I$ and
$f\in(I\colon\mathfrak{p})\setminus I$ such that $f\in S_v$. Then
$(I\colon f)=\mathfrak{p}$, $f\in\mathcal{F}_d$, and $\delta_I(d)=\deg(S/(I\colon f))=1$.
We claim that $f$ is not in $(I,h)$. If $f\in(I,h)$, then by part (a)
one has $v\geq 2$ and $\delta_I(v-1)=1$, a contradiction because
 $v$ is the regularity index of $\delta_I$ (see
Proposition~\ref{monday-morning}). Thus $f\notin(I,h)$. Next we show
the equality $(\mathfrak{p},h)=((I,h)\colon f)$. The inclusion
``$\subset$'' is clear because $(I\colon f)=\mathfrak{p}$. Take an
associated prime $\mathfrak{p}'$ of $((I,h)\colon f)$. The height of
$\mathfrak{p}'$ is ${\rm ht}(I)+1$ because $(I,h)$ is
Cohen--Macaulay. Then $\mathfrak{p}'=(\mathfrak{p}'',h)$ for some
$\mathfrak{p}''$ in ${\rm Ass}(I)$. Taking into account that
$\mathfrak{p}$ and $\mathfrak{p}''$ are generated by linear forms, we
get the equality $(\mathfrak{p},h)=(\mathfrak{p}'',h)$. Thus
$(\mathfrak{p},h)$ is equal to $((I,h)\colon f)$. Hence
$\delta_{(I,h)}(v)=1$,  and consequently
${\rm v}(I,h)={\rm reg}(\delta_{(I,h)})\leq {\rm reg}(\delta_I)={\rm
v}(I)=v$.

(c) There exists a system of parameters
$\underline{h}=h_1,\ldots,h_t$ of $S/I$ consisting of linear forms,
where $t=\dim(S/I)$. As $S/I$ is Cohen--Macaulay, $\underline{h}$ is a
regular sequence on $S/I$. Hence, by part (b), we obtain
$$
{\rm v}(I,\underline{h})={\rm v}(I,h_1,\ldots,h_t)\leq \cdots\leq
{\rm v}(I,h_1)\leq {\rm v}(I).
$$
\quad Thus, by Corollary~\ref{socle-vnumber-dim=0}, we get
${\rm s}(I)={\rm s}(I,\underline{h})=
\alpha(((I,\underline{h})\colon\frak{m})/(I,h))={\rm
v}(I,\underline{h})\leq{\rm v}(I)$.
\end{proof}

\subsection{Minimum distance function for Geramita ideals and Cayley-Bacharach ideals}

The minimum socle degree ${\rm s}(I)$, the local v-number ${\rm v}_\mathfrak{p}(I)$, and
the regularity ${\rm reg}(S/I)$, are
related below. For complete intersections of dimension $1$ they are
all equal. In particular in this case one has $\delta_I(d)\geq 2$ for
$1\leq d<{\rm reg}(S/I)$.

\begin{theorem}\label{banff-july-22-29-2018}
Let $I\subset S$ be a Geramita ideal and $\mathfrak{p}\in{\rm
Ass}(I)$. If $I$ is not prime, then
$$
{\rm s}(I)\leq {\rm v}_\mathfrak{p}(I) \leq{\rm reg}(S/I),
$$
with equality everywhere if $S/I$ is a level ring.
\end{theorem}

\begin{proof} We set $M=S/I$, $r_0={\rm reg}(S/I)$,
$n={\rm v}_\mathfrak{p}(I)$, and $I'=(I\colon\mathfrak{p})$.
To show the inequality $n\leq r_0$ we proceed by contradiction.
Assume that $n>r_0$. The $S$-modules in the exact sequence
$$0\longrightarrow I'/I\longrightarrow S/I \longrightarrow S/I' \longrightarrow 0$$
are nonzero Cohen--Macaulay modules of dimension $1$. Indeed, that $I'/I\neq
0$ (resp. $S/I'\neq 0$) follows from Proposition~\ref{lem:vnumber}
(resp. $I$ is not prime). That the modules are Cohen--Macaulay follows
observing that $I$ and $I'$ are unmixed ideals of dimension $1$.
Since $n$ is ${\rm v}_\mathfrak{p}(I)$ and $r_0<n$, one has
$(I'/I)_{r_0}=0$ (see the equivalence of Eq.~(\ref{oct8-18}) in the
proof of Proposition~\ref{lem:vnumber}). Hence taking Hilbert
functions in the above exact sequence in degree
$d=r_0$ (resp. for $d\gg 0$), by
Theorem~\ref{hilbert-function-dim=1},
we get
$$
\deg(S/I)=H_I(r_0)=H_{I'}(r_0)\leq\deg(S/I')\quad (\mbox{resp. }\
\deg(S/I)=\deg(I'/I)+\deg(S/I')).
$$
\quad As $I'/I\neq 0$, $\deg(I'/I)>0$. Hence $\deg(S/I)>\deg(S/I')$, a
contradiction. Thus $n\leq r_0$.

To show the inequality ${\rm s}(I)\leq {\rm v}_\mathfrak{p}(I)$ we
make a change of coefficients. Consider the algebraic
closure $\overline{K}$ of $K$. We set
$$
\overline{S}=S\otimes_K
\overline{K}=\overline{K}[t_1,\ldots,t_s]\ \mbox{ and }\ \overline{I}=I
\overline{S}.
$$
\quad Note that $K\hookrightarrow \overline{K}$ is a faithfully flat extension.
Apply the functor $S\otimes_{K}(-)$. By base change, it follows that
$S\hookrightarrow \overline{S}$ is a faithfully flat extension.
Therefore $H_I(d)=H_{\overline{I}}(d)$ for $d\geq 0$ and
$\deg(S/I)=\deg(\overline{S}/\overline{I})$. Furthermore the
minimal graded free resolutions and the Hilbert series of $S/I$ and
$\overline{S}/\overline{I}$ are identical. Thus $S/I$
and $\overline{S}/\overline{I}$ have the same regularity, ${\rm
s}(I)={\rm s}(\overline{I})$, and
$\overline{I}$ is Cohen--Macaulay of dimension $1$. The
ideal $\overline{\mathfrak{p}}=\mathfrak{p}\overline{S}$ is a prime
ideal of $\overline{S}$ because $\mathfrak{p}$ is generated by linear
forms, and so is $\overline{\mathfrak{p}}$. The ideal $\overline{I}$
is Geramita. To show this, let $I=\cap_{i=1}^m\mathfrak{q}_i$ be the
minimal primary decomposition of $I$, where $\mathfrak{q}_i$ is a
$\mathfrak{p}_i$-primary ideal. Since $\mathfrak{p}_i\overline{S}$ is
prime, the ideal $\mathfrak{q}_i\overline{S}$ is a
$\mathfrak{p}_i\overline{S}$-primary ideal of $\overline{S}$, and the
minimal primary decomposition of $\overline{I}$ is
$$
\overline{I}=\Bigg(\bigcap_{i=1}^m\mathfrak{q}_i\Bigg)\overline{S}=
\bigcap_{i=1}^m\big(\mathfrak{q}_i\overline{S}\big),
$$
see \cite[Sections~3.H, 5.D and 9.C]{Mat}. Thus $\overline{I}$ is a Geramita
ideal. Recall that $n\geq 1$ is the smallest integer
such that there is $f\in S_n$ with $(I\colon f)=\mathfrak{p}$. Fix
$f$ with these two properties. Then
$f\in(I\colon\mathfrak{p})\setminus I$ and since $\overline{I}\cap S=I$ and
$(I\colon\mathfrak{p})\overline{S}=(I\overline{S}\colon\mathfrak{p}\overline{S})$, one has
$f\in(I\overline{S}\colon\mathfrak{p}\overline{S})\setminus
I\overline{S}$. Therefore, setting
$\overline{\mathfrak{p}}=\mathfrak{p}\overline{S}$, we obtain
${\rm v}_{\overline{\mathfrak{p}}}(\overline{I})
\leq {\rm v}_{\mathfrak{p}}(I)$. Altogether using Proposition~\ref{s-v-numbers}(c), we obtain
$$
{\rm s}(I)={\rm s}(\overline{I})\leq
{\rm v}(\overline{I})\leq{\rm v}_{\overline{\mathfrak{p}}}(\overline{I})
\leq {\rm v}_{\mathfrak{p}}(I)\leq{\rm reg}(S/I)={\rm reg}(\overline{S}/\overline{I}).
$$
\quad If $S/I$ is level then so is $\overline{S}/\overline{I}$, because the Betti numbers ($b_{i,j}$ in Definition \ref{regularity-socle-degree}) for $S/I$ and $\overline{S}/\overline{I}$ agree \cite[6.10]{Eisen}.
Furthermore, since  the ring $\overline{S}/\overline{I}$ is level, we have ${\rm s}(\overline{I})={\rm reg}(\overline{S}/\overline{I})$ by \cite[4.13, 4.14]{eisenbud-syzygies} and which gives equality everywhere.
\end{proof}

\begin{definition}\cite{Guardo-Marino-Van-Tuyl,tohaneanu-vantuyl}\label{separator-def} Let
$Z=a_1P_1+\cdots+a_mP_m\subset\mathbb{P}^{s-1}$ be a set of fat
points, and suppose that $Z'=a_1P_1+\cdots+(a_i-1)P_i+\cdots+a_mP_m$
for some $i=1,\ldots,m$. We call $f\in S_d$ a \textit{separator of $P_i$ of
multiplicity} $a_i$ if $f\in I(Z')\setminus I(Z)$. The \textit{vanishing
ideal} $I(Z)$ of $Z$ is $\cap_{i=1}^m\mathfrak{p}_i^{a_i}$, where
$\mathfrak{p}_i$ is the vanishing ideal of $P_i$. If $Z$ is a set of
reduced points (i.e., $a_1=\dots=a_m=1$), the \textit{degree} of $P_i$,
denoted $\deg_Z(P_i)$, is the least degree of a separator of $P_i$ of multiplicity 1.
\end{definition}

\begin{remark}\label{separator-obs} If $f$ is a separator of $P_i$ of
multiplicity $a_i$
and $\mathfrak{p}_i$ is
the vanishing ideal of $P_i$, then $f\in (I\colon
\mathfrak{p}_i)\setminus I$. The converse hold if $a_i=1$.
\end{remark}

\begin{corollary}\cite[Theorem~3.3]{tohaneanu-vantuyl}
Let $Z=a_1P_1+\cdots+a_mP_m\subset\mathbb{P}^{s-1}$ be a set of fat
points, and suppose that $Z'=a_1P_1+\cdots+(a_i-1)P_i+\cdots+a_mP_m$
for some $i=1,\ldots,m$. If $f$ is a separator of $P_i$ of
multiplicity $a_i$, then $\deg(f)\geq {\rm v}(I)\geq {\rm s}(I)$.
\end{corollary}

\begin{proof} If $f$ is a separator of $P_i$ of multiplicity $a_i$
and $\mathfrak{p}_i$ be the vanishing ideal of $P_i$,
then $f\in(I\colon \mathfrak{p}_i)\setminus I$. Hence, by
Proposition~\ref{lem:vnumber} and
Theorem~\ref{banff-july-22-29-2018},
one has $\deg(f)\geq {\rm v}(I)\geq {\rm s}(I)$.
\end{proof}

A finite set $\mathbb{X}=\{P_1,\ldots,P_m\}$ of reduced points in
$\mathbb{P}^{s-1}$ is \textit{Cayley-Bacharach} if every hypersurface of degree less than
${\rm reg}(S/I(\mathbb{X}))$ which contains all but one
point of $\mathbb{X}$ must contain all the points of $\mathbb{X}$ or
equivalently if $\deg_\mathbb{X}(P_i)={\rm reg}(S/I(\mathbb{X}))$ for
all $i=1,\ldots,m$ \cite[Definition~2.7]{geramita-cayley-bacharach}.
Since $\deg_\mathbb{X}(P_i)={\rm v}_{\mathfrak{p}_i}(I)$, where
$\mathfrak{p}_i$ is the vanishing ideal of $P_i$,
by Theorem~\ref{banff-july-22-29-2018} one can extend this notion to Geramita ideals.

\begin{definition}\label{cayley-bacharach-def} A Geramita ideal $I\subset S$ is called
\textit{Cayley--Bacharach} if
${\rm v}_\mathfrak{p}(I)$ is equal to ${\rm reg}(S/I)$
for all $\mathfrak{p}\in{\rm Ass}(I)$.
\end{definition}

As the next result shows Cayley-Bacharach ideals are connected to Reed--Muller type codes
and to minimum distance functions.

\begin{corollary}\label{oct11-18} A Geramita ideal $I\subset S$ is Cayley--Bacharach
if and only if $${\rm reg}(\delta_I)={\rm v}(I)={\rm reg}(S/I).$$
\end{corollary}

\begin{proof} It follows from Proposition~\ref{monday-morning} and
Theorem~\ref{banff-july-22-29-2018}.
\end{proof}

There are some families of Reed--Muller type codes where the minimum
distance and its index of regularity are known
\cite{cartesian-codes,sorensen}. In these cases one can
determine whether or not the corresponding sets of points are
Cayley--Bacharach.

\begin{corollary} If $K=\mathbb{F}_q$ is a finite field and
$\mathbb{X}=\mathbb{P}^{s-1}$, then $I(\mathbb{X})$ is
Cayley--Bacharach.
\end{corollary}

\begin{proof} It follows from Corollary~\ref{oct11-18} because
according to \cite{sorensen} the regularity index of
$\delta_{I(\mathbb{X})}$ is equal to ${\rm reg}(S/I(\mathbb{X}))$.
\end{proof}

Next we give a lemma that allows comparisons between
the generalized minimum distances of ideals related by containment.

\begin{lemma}
\label{lem:Fcomparison}
If $I,I'$ are unmixed graded ideals of the same height
and $J$ is a graded ideal such that $I'=(I\colon J$), then
$\mathcal{F}_{d}(I')\subset \mathcal{F}_{d}(I)$ and
$$\deg(S/I')-\delta_{I'}(d)\leq \deg(S/I)-\delta_{I}(d).$$
\end{lemma}
\begin{proof}
Let $f\in \mathcal{F}_{d}(I')$. Then $f\notin I'$ and $(I'\colon
f)\neq I'$, and since we have the following relations
$$I'\subsetneq (I'\colon f)=((I\colon J)\colon f)=(I\colon
(fJ))=((I\colon f)\colon J)$$
we deduce that $(I\colon f)\neq I$ (otherwise the last ideal
displayed above would be $I'$). Note that $I\subset I'$, so $f\notin I$.
The second statement follows from the inequality
\begin{align*}
\deg(S/I')-\delta_{I'}(d) &=\max\{\deg(S/(I',f)) \mid
f\in \mathcal{F}_{d}(I')\}\\
&\leq \max\{\deg\left(S/(I,g)\right) \mid g\in \mathcal{F}_{d}(I)\}=  \deg(S/I)-\delta_{I}(d).
\end{align*}
\quad This inequality is a consequence of the observation that if
$f\in\mathcal{F}_{d}(I')$, then ${\rm ht}(I',f)={\rm ht}(I')$, and
since $f\in\mathcal{F}_{d}(I)$ one also has ${\rm ht}(I,f)={\rm
ht}(I)$ by Lemma~\ref{degree-initial-footprint}(a).
Thus $\deg(S/(I',f))\leq\deg(S/(I,f))$.
\end{proof}

One of our main results shows that the function
$\eta\colon\mathbb{N}_+\rightarrow\mathbb{Z}$
given by
$$
\eta(d):= (\deg(S/I)-H_I(d)+1)-\delta_I(d)
$$
non-negative for Geramita ideals (see Theorem~\ref{banff-theorem-1}).

\begin{lemma}\label{lemma-july-22-29-2018}
Let $I\subset S$ be a Geramita ideal.
If $\mathcal{F}_{d_0}=\emptyset$ for
some $d_0\geq 1$, then $\eta(d_0)=0$
and $\eta(d)\geq 0$ for all $d\geq 1$.
\end{lemma}

\begin{proof} Let $\mathfrak{p}_1,\ldots,\mathfrak{p}_m$ be the
associated primes of $I$. As $\mathfrak{p}_k$ is generated by
linear forms, the initial ideal of $\mathfrak{p}_k$, w.r.t the
lexicographical order $\prec$, is generated by
$s-1$ variables. Hence, as $\mathfrak{p}_k$ and ${\rm
in}_\prec(\mathfrak{p}_k)$ have the
same Hilbert function, $\deg(S/\mathfrak{p}_k)=1$ and
$H_{\mathfrak{p}_k}(d)=1$ for $d\geq 1$. Assume that
$\mathcal{F}_{d_0}=\emptyset$. Then
$\delta_{I}(d_0)=\deg(S/I)$ and $(I\colon f)=I$ for
any $f\in S_{d_0}\setminus I$. Hence, by Theorem~\ref{zero-divisors},
we get
$$
(\mathfrak{p}_1)_{d_0}\subset\left(\bigcup_{i=1}^m\mathfrak{p}_i\right)
\cap S_{d_0}
\subset I_{d_0}\subset(\mathfrak{p}_1)_{d_0}.
$$
\quad Thus $I_{d_0}=(\mathfrak{p}_1)_{d_0}$,
$H_I(d_0)=H_{\mathfrak{p}_1}(d_0)=1$, $H_I(0)=1$, and $\eta(d_0)=0$. Using
Theorem~\ref{hilbert-function-dim=1}(ii),
one has $H_I(d)=1$ for $d\geq 1$. Therefore $\eta(d)\geq 0$ for $d\geq 1$.
\end{proof}

\subsection{Singleton bound}
We come to one of our main results. The inequality in the following theorem is well known when $I$ is the vanishing ideal of a finite set of projective points \cite[p.~82]{algcodes}. In this case
the inequality is called the {\it Singleton bound\/}
\cite[Corollary~1.1.65]{tsfasman}. 

\begin{theorem}\label{banff-theorem-1}
Let $I\subset S$ be an unmixed graded ideal whose associated
primes
are generated by linear forms and such that there exists $h\in S_1$ regular on $S/I$. If
$\dim(S/I)=1$, then
$$
\delta_I(d)\leq \deg(S/I)-H_I(d)+1\ \mbox{ for }d \geq 1
$$
or equivalently $H_I(d)-1\leq {\rm hyp}_I(d)$ for $d\geq 1$.
\end{theorem}

\begin{proof}
The proof is by induction on $\deg(S/I)$.
If $\deg(S/I)=1$, then $I=\mathfrak{p}$ is a prime generated by
linear forms, $H_I(d)=1$ for all $d\geq 0$ and
$\mathcal{F}_{d}(I)=\emptyset$ for all $d\geq 1$. The latter
follows since for any prime $\mathfrak{p}$, $(\mathfrak{p}\colon f)\neq
\mathfrak{p}$ implies $f\in \mathfrak{p}$. So the result is verified
in this case. Let $v={\rm v}(I)$ be the v-number of $I$. By
Proposition~\ref{lem:vnumber},
$v=\alpha((I\colon \mathfrak{p})/I)$ for some
 $\mathfrak{p}\in {\rm Ass}(I)$. Set $I'=(I\colon \mathfrak{p})$. The short exact sequence
$$0\longrightarrow I'/I\longrightarrow S/I \longrightarrow S/I' \longrightarrow 0$$
together with the unmixed property of $S/I$ show that
$\dim\left(I'/I\right)=1$ and ${\rm depth}(I'/I)=1$. Therefore,
$H_{I'/I}(d) =0$ for $d<\alpha((I\colon \mathfrak{p})/I)=v$ and
$H_{I'/I}(d) >0$ for $d\geq \alpha((I\colon \mathfrak{p})/I)=v$, and
consequently $H_{I'}(d)=H_{I}(d)$ for $d<v$ and $H_{I'}(d)>H_{I}(d)$
for $d\geq v$.
The last statement yields that $\deg(S/I)>\deg(S/I')$. This also
follows from Lemma~\ref{tohaneanu-aug29-18}(b).

If $d<v$ we deduce from Lemma \ref{lem:Fcomparison},
the inductive hypothesis  and $H_{I'}(d)=H_{I}(d)$ that
$$\deg(S/I)-\delta_I(d)\geq \deg(S/I')-\delta_{I'}(d)\geq H_{I'}(d)-1=H_I(d)-1,$$
which is the desired inequality. If $d\geq v$ we know that there
exists $f\in S_v$ such that $(I\colon f)=\mathfrak{p}$ and thus
$(I\colon h^{d-v}f)=\mathfrak{p}$. Therefore $\delta_{I}(d)=1$ and
since $\deg(S/I)\geq H_I(d)$ for any $d$ the desired inequality follows.
\end{proof}

The next result is known for complete intersection vanishing ideals
over finite fields \cite[Lemma~3]{sarabia7}. As an
application we extend this result to Geramita Gorenstein ideals.

\begin{corollary}\label{geramita-gorenstein} Let $I\subset S$ be a
Geramita ideal.
If $I$ is Gorenstein and $r_0={\rm reg}(S/I)\geq 2$, then
$\delta_I(r_0-1)$ is equal to $2$.
\end{corollary}

\begin{proof} 
By Proposition~\ref{monday-morning} and
Theorem~\ref{banff-july-22-29-2018}, $r_0$ is the regularity index of
$\delta_I$.Thus $\delta_I(r_0-1)\geq 2$. 

We show that $\deg(S/I)=1+H_I(r_0-1)$. For this, we may assume that $K$ is infinite. 
Indeed, consider the algebraic closure $\overline{K}$ of $K$. We set
$\overline{S}=S\otimes_K \overline{K}$
and 
$ \overline{I}=I \overline{S}.$
\quad From \cite[Lemma 1.1]{Sta1}, we have $\overline{I}$ is Gorenstein, $H_I(d)=H_{\overline{I}}(d)$ for $d\geq 0$ and
$\deg(S/I)=\deg(\overline{S}/\overline{I})$. Since $\overline{K}$ is infinite, there is $h\in \overline{S}$ that is regular on $\overline{S}/\overline{I}$. Then by \cite[3.1.19]{BHer}(b) the quotient ring $A=\overline{S}/(\overline{I},h)$ is Gorenstein of dimension 0, by \cite[4.13, 4.14]{eisenbud-syzygies} it has $r_0={\rm reg}(\overline{S}/\overline{I})={\rm reg}(A)$,  and   by \cite[4.7.11(b)]{BHer} $\deg(\overline{S}/\overline{I})=\deg(A)=\sum_{i=0}^{r_0}H_A(i)$. By  \cite[proof of 4.1.10]{BHer}, $F_A(x)=(1-x)F_I(x)$ hence $H_{\overline{I}}(n)= \sum_{i=0}^n H_A(i)$ for any $n\geq 0$.
From here, using that $H_A(r_0)=1$, since $A$ is Gorenstein, we deduce that
$$\deg(S/I)=\deg(\overline{S}/\overline{I})=\deg(A)=\sum_{i=0}^{r_0}H_A(i)=1+\sum_{i=0}^{r_0-1}H_A(i)=1+H_{\overline{I}}(r_0-1)=1+H_I(r_0-1).$$
Finally,  making $d=r_0-1$ in Theorem~\ref{banff-theorem-1}, we get $\delta_I(r_0-1)\leq 2$. Thus equality holds.
\end{proof}

Note that the situation is quite different from the conclusion of Theorem \ref{banff-theorem-1} if $\dim(S/I)\geq 2$.

\begin{proposition}\label{sep10-18} Let $I\subset S$ be an unmixed graded ideal. If
$\dim(S/I)\geq 2$, then
$$
\delta_I(d)>\deg(S/I)-H_I(d)+1\ \mbox{ for some }d \geq 1.
$$
\end{proposition}

\begin{proof} Note that
$\mathfrak{m}=(t_1,\ldots,t_s)$ is not an associated prime of $I$,
that is, ${\rm depth}(S/I)\geq 1$. Assume
that $\mathcal{F}_d=\emptyset$ for some $d\geq 2$. As $H_I(0)=1$ and
$\delta_I(d)$ is equal to $\deg(S/I)$, by
Theorem~\ref{hilbert-function-dim=1}(i), one has $H_I(d)>1$ and the
inequality holds. Now assume that $\mathcal{F}_d\neq\emptyset$ for
$d\geq 2$. For each $d\geq 2$ pick $f_d\in\mathcal{F}_d$ such that
$$\delta_I(d)=\deg(S/I)-\deg(S/(I,f_d)).$$
\quad As $H_I$ is strictly increasing by
Theorem~\ref{hilbert-function-dim=1}(i), using
Lemma~\ref{degree-initial-footprint}(b), we get
$$
\deg(S/(I,f_d))<\deg(S/I)<H_I(d)-1
$$
for $d\gg 0$. Thus the required inequality holds for $d\gg 0$.
\end{proof}

\section{Reed-Muller type codes}\label{Reed-Muller-section}

In this section we give refined information on the minimum distance function for the Reed--Muller codes defined in the Introduction. The key insight is that, in the case of the projective Reed--Muller codes, this minimum distance function can be realized as a generalized minimum distance function for a finite set of points in projective space, often called evaluation points in the algebraic coding context.

\begin{theorem}{\cite[Theorem~4.5]{rth-footprint}}\label{jul11-18}
Let $\mathbb{X}$ be a finite set of points in a
projective space $\mathbb{P}^{s-1}$ over a field $K$ and let
$I(\mathbb{X})$ be its
vanishing ideal. If $d\geq 1$ and $1\leq r\leq H_\mathbb{X}(d)$, then
$$
\delta_r(C_\mathbb{X}(d))=\delta_{I(\mathbb{X})}(d,r).
$$
\end{theorem}

By Theorem \ref{jul13-18} (a) and \cite[Theorem 1, Corollary 1]{wei}, the entries of each row of the weight matrix
$(\delta_{\mathbb{X}}(d,r))$ form an increasing sequence until they
stabilize. We show in Theorem \ref{summer-18} below that the entries of
each column of the weight matrix $(\delta_{\mathbb{X}}(d,r))$ form a decreasing sequence.

Before we can prove this result we need an additional lemma.
Recall that the {\it support\/} $\chi(\beta)$ of a vector $\beta\in K^m$
is $\chi(K\beta)$, that is, $\chi(\beta)$ is the set of non-zero
entries of $\beta$.

\begin{lemma}\label{seminar} Let $D$ be a subcode of $C$ of dimension $r\geq 1$. If
$\beta_1,\ldots,\beta_r$ is a $K$-basis for $D$ with
$\beta_i=(\beta_{i,1},\ldots,\beta_{i,m})$ for $i=1,\ldots,r$, then
$\chi(D)=\cup_{i=1}^r\chi(\beta_i)$ and the number of elements of
$\chi(D)$ is the number of non-zero columns of the matrix:
$$
\left[\begin{matrix}
\beta_{1,1}&\cdots&\beta_{1,i}&\cdots&\beta_{1,m}\\
\beta_{2,1}&\cdots&\beta_{2,i}&\cdots&\beta_{2,m}\\
\vdots&\cdots&\vdots&\cdots&\vdots\\
\beta_{r,1}&\cdots&\beta_{r,i}&\cdots&\beta_{r,m}
\end{matrix}\right].
$$
\end{lemma}

\begin{theorem}\label{summer-18} Let $\mathbb{X}$ be a finite
set of points in $\mathbb{P}^{s-1}$, let $I=I(\mathbb{X})$ be its vanishing ideal, and let
$1\leq r\leq |\mathbb{X}|$ be a fixed integer. Then there is an integer $d_0\geq 1$ such that
$$
\delta_I(1,r)>\delta_I(2,r)>\cdots> \delta_I(d_0,r)=\delta_I(d,r)=r\ \mbox{ for }\
d\geq d_0.
$$
\end{theorem}

\begin{proof} Let $[P_1],\ldots,[P_m]$ be the points of $\mathbb{X}$.
By Theorem~\ref{jul11-18} there exists a linear subcode $D$ of $C_\mathbb{X}(d)$ of
dimension $r$ such that $\delta_I(d,r)=\delta_\mathbb{X}(d,r)=|\chi(D)|$. Pick a
$K$-basis $\beta_1,\ldots,\beta_r$ of $D$. Each $\beta_i$ can be
written as
$$
\beta_i=(\beta_{i,1},\ldots,\beta_{i,k},\ldots,\beta_{i,m})=
(f_i(P_1),\ldots,f_i(P_k),\ldots,f_i(P_m))
$$
for some $f_i\in S_d$. Consider the matrix $B$ whose rows are
$\beta_1,\ldots,\beta_m$:
$$
B=\left[\begin{matrix}
f_1(P_1)&\cdots&f_1(P_k)&\cdots&f_1(P_m)\\
f_2(P_1)&\cdots&f_2(P_k)&\cdots&f_2(P_m)\\
\vdots&\cdots&\vdots&\cdots&\vdots\\
f_r(P_1)&\cdots&f_r(P_k)&\cdots&f_r(P_m)
\end{matrix}\right].
$$
\quad As $B$ has rank $r$, by permuting columns and applying elementary
row operations, the matrix $B$ can be brought to the form:
$$
B'=\left[\begin{matrix}
g_1(Q_1)& & & &g_1(Q_{r+1})&\cdots&g_1(Q_m)\\
&g_2(Q_2)& &\mathbf{0}\ &g_2(Q_{r+1})&\cdots&g_2(Q_m)\\
\mathbf{0} & &\ddots & &\vdots\\
 & & &g_r(Q_r)&g_r(Q_{r+1})&\cdots& g_r(Q_m)
\end{matrix}\right],
$$
where $g_1,\ldots,g_r$ are linearly independent polynomials over the
field $K$
modulo $I$ of degree $d$, $Q_1,\ldots,Q_m$ are a permutation of
$P_1,\ldots,P_m$, the
first $r$ columns of $B'$ form a diagonal matrix
such that $g_i(Q_i)\neq 0$ for $i=1,\ldots,r$, and the ideals
$(f_1,\ldots,f_r)$ and $(g_1,\ldots,g_r)$ are equal. Let
$D'$ be the linear space generated by the rows of $B'$. The operations
applied to $B$ did not affect the size of the support of $D$
(Lemma~\ref{seminar}), that is, $|\chi(D)|=|\chi(D')|$.

Note that $\delta_r(C_\mathbb{X}(d))$
depends only on $\mathbb{X}$, that is, $\delta_r(C_\mathbb{X}(d))$ is
independent of how we order the points in $\mathbb{X}$ (cf.
Theorem~\ref{jul11-18}). Let ${\rm ev}_{d}'\colon S_d\rightarrow K^m$
be the evaluation map,
$f\mapsto (f(Q_1),\ldots,f(Q_m))$,
relative to the points $[Q_1],\ldots,[Q_m]$.
By Theorem~\ref{jul13-18}, $\delta_\mathbb{X}(d,r)\geq r$.

First we assume that
$\delta_\mathbb{X}(d,r)=r$ for some $d\geq 1$ and $r\geq 1$.
Then the $i$-th column of $B'$ is zero for $i>r$. For each $1\leq
i\leq r$ pick $h_i\in S_1$ such that $h_i(Q_i)\neq 0$. The
polynomials $h_1g_1,\ldots,h_rg_r$ are linearly independent
modulo $I$ because $(h_ig_i)(Q_j)$ is not $0$ if $i=j$ and is $0$
if $i\neq j$. The image
of $Kh_1g_1\oplus\cdots\oplus Kh_rg_r$, under the map ${\rm
ev}'_{d+1}$, is a subcode $D''$ of $C_\mathbb{X}(d+1)$ of dimension $r$ and
$|\chi(D'')|=r$. Thus $\delta_\mathbb{X}(d+1,r)\leq r$, and
consequently $\delta_\mathbb{X}(d+1,r)=r$.

Next we assume that
$\delta_\mathbb{X}(d,r)>r$. Then $B'$ has a nonzero column
$(g_1(Q_k),\ldots,g_r(Q_k))^\top$ for some $k>r$. It suffices to show that
$\delta_\mathbb{X}(d,r)>\delta_\mathbb{X}(d+1,r)$.  According to
\cite[Lemma~2.14(ii)]{hilbert-min-dis} for each $1\leq i\leq r$ there
is $h_i$ in $S_1$ such that $h_i(Q_i)\neq 0$ and $h_i(Q_k)=0$.
Let $B''$ be the matrix:
$$
B''=\left[\begin{matrix}
h_1g_1(Q_1)& & & &h_1g_1(Q_{r+1})&\cdots&h_1g_1(Q_m)\\
&h_2g_2(Q_2)& &\mathbf{0}\ &h_2g_2(Q_{r+1})&\cdots&h_2g_2(Q_m)\\
\mathbf{0} & &\ddots & &\vdots\\
 & & &h_rg_r(Q_r)&h_rg_r(Q_{r+1})&\cdots& h_rg_r(Q_m)
\end{matrix}\right].
$$
\quad The image
of $Kh_1g_1\oplus\cdots\oplus Kh_rg_r$, under the map ${\rm
ev}'_{d+1}$, is a subcode $V$ of $C_\mathbb{X}(d+1)$ of
dimension $r$ because the rank of $B''$ is $r$, and since the $k$-column of
$B''$ is zero, we get
$$
\delta_\mathbb{X}(d,r)=|\chi(D)|=|\chi(D')|>|\chi(V)|\geq
\delta_\mathbb{X}(d+1,r).
$$
\quad Thus $\delta_\mathbb{X}(d,r)>\delta_\mathbb{X}(d+1,r)$.
\end{proof}

\begin{corollary}\label{sept1-18} Let $\mathbb{X}$ be a finite
set of points in $\mathbb{P}^{s-1}$ and let $I=I(\mathbb{X})$ be its
vanishing ideal. If $I$ is a complete intersection, then
$\delta_I(d)\geq {\rm reg}(S/I)-d+1$ for $1\leq d<{\rm reg}(S/I)$.
\end{corollary}

\begin{proof} If $r_0$ denotes the regularity of $S/I$,
by Theorem~\ref{banff-july-22-29-2018}, one has
${\rm v(I)}=r_0$. Thus $\delta_I(r_0-1)\geq 2$ and the result follows
from Theorem~\ref{summer-18} by setting $r=1$.
\end{proof}

\begin{corollary}{\cite[Theorem~12]{camps-sarabia-sarmiento-vila}}
If $\mathbb{X}$ is a set parameterized by monomials lying on a
projective torus and $1\leq r\leq|\mathbb{X}|$ be a fixed integer,
then there is an integer $d_0\geq 1$ such that
$$
\delta_r(C_\mathbb{X}(1))>\delta_r(C_\mathbb{X}(2))>\cdots>
\delta_r(C_\mathbb{X}(d_0))=\delta_r(C_\mathbb{X}(d))=r\ \mbox{ for }\
d\geq d_0.
$$
\end{corollary}

\begin{proof} It follows at once from Theorems~\ref{jul11-18} and
\ref{summer-18}.
\end{proof}

\begin{corollary}\label{min-dis-v-number} Let $\mathbb{X}$ be a finite set of points of
$\mathbb{P}^{s-1}$ and let $\delta_\mathbb{X}(d)$ be the minimum
distance of $C_\mathbb{X}(d)$. Then $ \delta_\mathbb{X}(d)=1$ if and
only if $d\geq {\rm v}(I)$.
\end{corollary}

\begin{proof} It follows from Proposition~\ref{monday-morning}, and
Theorems~~\ref{jul11-18} and \ref{summer-18}.
\end{proof}

\begin{corollary}\label{jul15-18-coro} If $\mathbb{X}$ is a finite
set of $\mathbb{P}^{s-1}$ over a field $K$, then
$\delta_\mathbb{X}(d,H_\mathbb{X}(d))=|\mathbb{X}|$ for $d \geq 1$.
\end{corollary}

\begin{proof} It follows at once from Proposition~\ref{jul15-18} and
Theorem~\ref{jul11-18}.
\end{proof}

\section{Complete intersections}\label{ci-section}
In this section we examine minimum distance functions of complete
intersection ideals. 

\begin{definition}\label{ci-def} An ideal $I\subset S$ is called a {\it complete intersection\/} if
there exist $g_1,\ldots,g_{r}$ in $S $ such that $I=(g_1,\ldots,g_{r})$,
where $r={\rm ht}(I)$ is the height of $I$.
\end{definition}

There are a number of interesting open problems regarding the minimum distance of complete intersection functions. We discuss one such problem in Conjecture \ref{footprint-ci-general} below and relate this problem to \cite[Conjecture~CB12]{Eisenbud-Green-Harris}) in the second part of this section.

\begin{conjecture}{\cite{min-dis-ci}}\label{footprint-ci-general}
Let $I\subset S:=K[t_1, \ldots, t_s]$ be a complete intersection graded ideal of dimension
$1$ generated by forms $f_1,\ldots,f_c$, $c=s-1$, with
$d_i=\deg(f_i)$ and $2\leq d_i\leq d_{i+1}$ for $i\geq 1$. If the associated
primes of $I$ are generated by linear forms, then
$$
\delta_I(d)\geq (d_{k+1}-\ell)d_{k+2}\cdots d_c\ \mbox{ if }\
1\leq d\leq \sum\limits_{i=1}^{c}\left(d_i-1\right)-1,
$$
where $0\leq k\leq c-1$ and $\ell$ are integers such that
$d=\sum_{i=1}^{k}\left(d_i-1\right)+\ell$ and $1\leq \ell \leq
d_{k+1}-1$.
\end{conjecture}

This conjecture holds if the initial ideal of $I$ with respect to some monomial order is a complete
intersection \cite[Theorem~3.14]{min-dis-ci}. 
Our  results show
that for complete intersections ${\rm v}(I)={\rm reg}(S/I)$
(Theorem~\ref{banff-july-22-29-2018}) and
$\delta_I(d)\geq {\rm reg}(S/I)-d+1$ for $1\leq d<{\rm reg}(S/I)$
if $I$ is a vanishing ideal (Corollary~\ref{sept1-18}).
Thus the conjecture is best possible for vanishing ideals in the sense that it covers all cases where
$\delta_I (d) > 1$ because the regularity of $S/I$ is equal to
$\sum\limits_{i=1}^{c}\left(d_i-1\right)$.

Two special cases of the conjecture that are still open are the
following.

\begin{conjecture}\label{tohaneanu-eisenbud} Let $\mathbb{X}$ be a
finite set of reduced points in $\mathbb{P}^{s-1}$ and suppose that $I=I(\mathbb{X})$ is a
complete intersection generated by $f_1,\ldots,f_{c}$, $c=s-1$, with
$d_i=\deg(f_i)$ for $i=1,\ldots,c$, and $2\leq d_i\leq d_{i+1}$ for
all $i$. Then
\begin{itemize}
\item[(a)] \cite[Conjecture~4.9]{tohaneanu-vantuyl}
$\delta_I(1)\geq (d_1-1)d_2\cdots d_{c}$.
\item[(b)] If $f_1,\ldots,f_c$ are quadratic forms, then $\delta_I(d)\geq
2^{c-d}$ for $1\leq d\leq c$ or equivalently ${\rm hyp}_I(d)\leq
2^c-2^{c-d}$ for $1\leq d\leq c$.
\end{itemize}
\end{conjecture}

We prove part (a) of this conjecture, in a more general setting, when
$I$ is equigenerated.

\begin{proposition}\label{stefan-adam-uniform} Let $I\subset S$ be an
unmixed
graded ideal of
height $c$, minimally generated by forms of degree $e\geq 2$, whose associated
primes are generated by linear forms. Then
$$
{\rm hyp}_I(1)\leq e^{c-1}
$$
and $\delta_I(1)\geq \deg(S/I)-e^{c-1}$. Furthermore
$\delta_I(1)\geq e^c-e^{c-1}$ if $I$ is a complete intersection.
\end{proposition}

\begin{proof} Since the associated primes of $I$ are generated by linear
forms and $e\geq 2$, one has $\mathcal{F}_1(I)\neq\emptyset$.
Take any linear form $h=t_k-\sum_{j\neq
i}\lambda_jt_j$ in $\mathcal{F}_1(I)$, $\lambda_j\in K$. For
simplicity of notation assume $k=1$. It
suffices to show that $\deg(S/(I,h))\leq e^{c-1}$. Let
$\{f_1,\ldots,f_n\}$ be a minimal set of generators of $I$ consisting
of homogeneous polynomials with $\deg(f_i)=e$ for all $i$. Setting
$f_i'=f_i(\sum_{j\neq 1}\lambda_jt_j,t_2,\ldots,t_s)$ for
$i=1,\ldots,n$, $S'=K[t_2,\ldots,t_s]$, and $I'=(f_1',\ldots,f_n')$,
there is an isomorphism $\varphi$ of graded $K$-algebras
$$
S/(I,h)\stackrel{\varphi}{\longrightarrow} S'/I',\quad t_1\mapsto
\lambda_2t_2+\cdots+\lambda_st_s,\ \ t_i\mapsto t_i,\ i=2,\ldots,s.
$$
\quad Note that
$\varphi(f+(I,h))=f(\lambda_2t_2+\cdots+\lambda_st_s,t_2,\ldots,t_s)+I'$
for $f$ in $S$ and that $\varphi$ has degree $0$, that is, $\varphi$
is degree preserving. Hence $S/(I,h)$ and $\deg(S'/I')$ have the same
degree and the same dimension. Since ${\rm ht}(I,h)={\rm ht}(I)$, we
get ${\rm ht}(I')={\rm ht}(I)-1$, that is, ${\rm ht}(I')=c-1$. By
definition $f_i'$ is either $0$ or has degree $e$, that is, $I'$ is
generated by forms of degree $e$. As $K$ is infinite,
there exists a minimal set of generators of $I'$,
$\{g_1,\ldots,g_t\}$, such that $\deg(g_i)=e$ for all $i$ and
$g_1,\ldots,g_{c-1}$ form a regular sequence (see
Lemma~\ref{uniform-regular-sequence}).
From the exact sequence
$$0\longrightarrow I'/(g_1,\ldots,g_{c-1})\longrightarrow S'/(g_1,\ldots,g_{c-1})
\longrightarrow S'/I' \longrightarrow 0,$$
we get $e^{c-1}=\deg(S/(g_1,\ldots,g_{c-1}))\geq
\deg(S'/I')=\deg(S/(I,h))$. This
proves that ${\rm hyp}_I(1)$ is less than or equal to $e^{c-1}$. Hence $\delta_I(1)\geq
\deg(S/I)-e^{c-1}$. Therefore, if $I$ is a complete intersection,
$\deg(S/I)=e^c$ and we obtain the inequality $\delta_I(1)\geq
e^c-e^{c-1}$.
\end{proof}

As a consequence, we recover the fact that
Conjecture~\ref{tohaneanu-eisenbud}(a) holds for $\mathbb{P}^2$
\cite[Theorem~4.10]{tohaneanu-vantuyl}.

\begin{corollary}\label{stefan-adam-P2} Let $I\subset S$ be a
graded ideal of height $2$, minimally generated by two forms
$f_1,f_2$ of degrees $e_1,e_2$,
with $2\leq e_1\leq e_2$, whose associated
primes are generated by linear forms. Then ${\rm hyp}_I(1)\leq e_2$
and $\delta_I(1)\geq e_1e_2-e_2$.
\end{corollary}

\begin{proof} It follows adapting the proof of
Proposition~\ref{stefan-adam-uniform}.
\end{proof}

\begin{corollary}\label{stefan-adam-uniform-gen} Let $I\subset S$ be an unmixed graded ideal minimally
generated by
forms of degree $e\geq 2$ whose associated
primes are generated by linear forms.
If $1\leq r\leq {\rm ht}(I)$, then
$$
{\rm hyp}_I(1,r)\leq e^{{\rm ht}(I)-r},
$$
$\delta_I(1,r)\geq \deg(S/I)-e^{{\rm ht}(I)-r}$,
and $\delta_I(1,r)\geq e^{{\rm ht}(I)}-e^{{\rm ht}(I)-r}$ if $I$ is a
complete intersection. 
\end{corollary}
\begin{proof}
This follows by adapting the proof of
Proposition~\ref{stefan-adam-uniform} and observing the following. If
$f_1,\ldots,f_r$ are linearly independent linear forms and $t_1\succ\cdots\succ t_s$ is the
lexicographical order, we can find linear forms $h_1,\ldots,h_r$ such
that ${\rm in}_\prec(h_1)\succ\cdots\succ{\rm in}_\prec(h_r)$ and
$(f_1,\ldots,f_r)$ is equal to $(h_1,\ldots,h_r)$.
\end{proof}

\subsection*{Cayley-Bacharach Conjectures} In the following we explore the connections between
a modified form of Conjecture~\ref{footprint-ci-general} and a conjecture of Eisenbud-Green-Harris  \cite[Conjecture~CB12]{Eisenbud-Green-Harris}.

\begin{conjecture}[Strong form of {\cite[Conjecture~CB12]{Eisenbud-Green-Harris}} ]
\label{CB12}
Let $\Gamma$ be any subscheme of a zero-dimensional complete intersection
of hypersurfaces of degrees $d_1\leq \dots \leq d_c$ in a projective space $P^c$. If $\Gamma$ fails to impose independent conditions on hypersurfaces of degree $m$, then
$$\deg(\Gamma)\geq (e+1)d_{k+2}d_{k+3} \cdots d_c$$ where $e$ and $k$ are defined by the relations
$$\sum_{i=k+2}^c (d_i-1) \leq m+1 < \sum_{i=k+1}^c (d_i-1) \quad \text{ and } \quad e=m+1-\sum_{i=k+2}^c (d_i-1).$$
\end{conjecture}

\begin{proposition}
 Conjectures \ref{footprint-ci-general} and \ref{CB12} are equivalent for radical complete intersections.
 \end{proposition}

\begin{proof} 
We  first prove that Conjecture \ref{CB12} for $m=\sum_{i=k+1}^c(d_i-1)-\ell-1$ and $e=d_{k+1}-\ell-1$ implies Conjecture \ref{footprint-ci-general}.
Let $I$ be a radical complete intersection ideal minimally generated by forms of degrees $d_1\leq \cdots \leq d_c$.
Let $H$ be any hypersurface defined by a form $F$ of degree $d$. Let $\mathbb{X}$ be the scheme defined by $I(\mathbb{X})=(I,F)$ and let $\Gamma$ be the residual scheme defined by $I(\Gamma)=I:F$. 
By the Cayley-Bacharach Theorem \cite[CB7]{Eisenbud-Green-Harris},  $\Gamma$ must fail to impose independent conditions on hypersurfaces of degree $\sum_{i=1}^c(d_i-1)-d-1=m$. Now  Conjecture \ref{CB12} implies $\deg(S/I:F)=\deg(\Gamma)\geq ed_{k+2}d_{k+3} \cdots d_c$, which in view of Theorem \ref{rth-min-dis-vi} gives 
$$\delta_I(d)\geq (e+1)d_{k+2}d_{k+3} \cdots d_c=(d_{k+1}-\ell)d_{k+2}d_{k+3} \cdots d_c.$$

For the converse, we prove that Conjecture \ref{footprint-ci-general} with $d=\sum_{i=1}^c(d_i -1)-m-1$ and $\ell=d_{k+1}-e-1$ recovers Conjecture \ref{CB12}. Let $\Gamma$ be any subscheme of a complete intersection, and suppose that $\Gamma$ fails to impose independent conditions on hypersurfaces of degree $m$. Assuming that $\Gamma$ spans a projective space
$\mathbb{P}^c$, take a radical complete intersection ideal $I$ contained in $I_\Gamma$, and let $\mathbb{X}$ be the scheme defined by $I(\mathbb{X})=I:I(\Gamma)$. By \cite[CB7]{Eisenbud-Green-Harris}, $\mathbb{X}$ lies on a hypersurface of degree $\sum_{i=1}^c(d_i -1)-m-1=d$. Then Conjecture \ref{footprint-ci-general} and Theorem \ref{rth-min-dis-vi} give 
 $$\deg(\Gamma)=\deg(S/I:F)\geq (d_{k+1}-\ell)d_{k+2}d_{k+3}\cdots d_c=(e+1)d_{k+2}d_{k+3}\cdots d_c.$$
\end{proof}

Conjecture \ref{CB12} has been recently proven in \cite[Theorem 5.1]{HU}  for $k=1$ under additional assumptions on the Picard group of the complete intersection.  We now consider the case when $d_1=\dots=d_c=2$. In this case Conjecture \ref{footprint-ci-general} specializes to  Conjecture \ref{tohaneanu-eisenbud}(b) and Conjecture \ref{CB12} is related to {\cite[Conjecture~CB10]{Eisenbud-Green-Harris}.

\begin{proposition}
\label{prop:CB10}
The following statements are equivalent:
\begin{enumerate}
\item {[Conjecture \ref{tohaneanu-eisenbud}(b)]} Let $I$ be a complete intersection generated by $c$ quadratic forms. Then $\delta_I(d)\geq 2^{c-d}$ for $1\leq d\leq c$ or equivalently ${\rm hyp}_I(d)\leq
2^c-2^{c-d}$ for $1\leq d\leq c$.
\item {\cite[Conjecture~CB10]{Eisenbud-Green-Harris}}
If $\mathbb{X}$ is an ideal-theoretic 
complete intersection of $c=s-1$ quadrics in $\mathbb P^{s-1}$ and $f \in
S:=K[t_1,\ldots,t_s]$ is a homogeneous polynomial of degree $d$ such
that $\deg(S/(I(\mathbb{X}),f))>2^c-2^{c-d}$, then $f\in I_{\mathbb{X}}$.
\end{enumerate}
\end{proposition}

\begin{proof}
Let $I=I(\mathbb{X})$ be a complete intersection ideal of $c$ quadratic homogeneous polynomials. Then  $\deg(S/I)=2^c$ and (2) is equivalent to the statement for any $f\in \mathcal{F}_d$, $\deg(S/(I(\mathbb{X}),f))\leq 2^c-2^{c-d}$. Using Definition \ref{def:GMD}, this is in turn equivalent to $\delta_I(d)\geq 2^{c-d}, \mbox{ for }d=1,\ldots,c$, which is precisely the statement of (1). 
\end{proof}

As an application of our earlier results we recover the following cases of
Conjecture~\ref{tohaneanu-eisenbud}(b) under the more general hypothesis that $I(\mathbb{X})$ is a not necessarily a radical complete intersection.

\begin{corollary}{\cite{Eisenbud-Green-Harris}}\label{conjecture-some-cases}
If $I$ is a complete intersection ideal generated by $c$ quadratic forms, then 
$\delta_I(d)\geq 2^{c-d}$, for $d=1, c-1,$ and $c$
\end{corollary}
\begin{proof}
Case $d=1$ follows by taking $e=2$ in Proposition~\ref{stefan-adam-uniform}.  Case $d=c-1$ follows from Corollary~\ref{geramita-gorenstein} because complete intersections are Gorenstein and vanishing ideals of finite sets of points are Geramita.  For Case $d=c$ recall that By Proposition~\ref{monday-morning} and Theorem~\ref{banff-july-22-29-2018}, ${\rm reg}(S/I)$ is the regularity index of
$\delta_I$. Thus Case $d=c$ follows since ${\rm reg}(S/I)=c$.
\end{proof}

If $\mathbb{X}\subset \mathbb P^c$ is a reduced set of points such that $I(\mathbb{X})$ is a complete intersection ideal and the points of $\mathbb{X}$ are in linearly general position, i.e., any $c+1$ points of $\mathbb{X}$ span $\mathbb P^c$, then \cite[Theorem 1]{BaFo} adds more cases when {\cite[Conjecture~CB10]{Eisenbud-Green-Harris}}  holds. If we assume that the quadrics that cut out $\mathbb{X}$ are generic, then the assumptions of that result are satisfied, and we can conclude by Proposition \ref{prop:CB10} that if $1\leq d\leq c-1$, then $\delta_I(d)\geq c(c-1-d)+2$.  Suppose $d:=c-k$, for some $1\leq k\leq c-1$. Then this inequality becomes $\delta_I(c-k)\geq c(k-1)+2$.  We observe that this inequality is sufficiently strong to establish [Conjecture \ref{tohaneanu-eisenbud}(b)] {\em asymptotically} for $c$ sufficiently large:
\begin{itemize}
  \item $k=2$. If $c\geq 2$, then $\delta_I(c-2)\geq c+2\geq 2^2$.
  \item $k=3$. If $c\geq 3$, then $\delta_I(c-3)\geq 2c+2\geq 2^3$.
  \item $k=4$. If $c\geq 5$, then $\delta_I(c-4)\geq 3c+2\geq 2^4$.
  \item $k\geq 5$. If $c\geq (2^k-2)/(k-1)$, then $\delta_I(c-k)\geq 2^k$.
\end{itemize}

%
%

\vspace{-0.5em}

\section*{Acknowledgements.} 
\vspace{-0.5em}
We thank the Banff International Research Station (BIRS) where this project was started during under the auspices of their focused research groups program. Computations with {\it Macaulay2} \cite{mac2} were crucial to verifying many of the examples given in this paper. The authors acknowledge the support of the following funding agencies: Cooper was partially funded by NSERC grant RGPIN-2018-05004, Seceleanu was partially supported by NSF grant DMS--1601024 and EPSCoR grant OIA--1557417.


\begin{appendix}
\label{Appendix}

\section{Procedures for {\it Macaulay\/}$2$}

\begin{procedure}\label{procedure-footprint-matrix} Computing the
footprint matrix with {\it Macaulay\/}$2$ \cite{mac2}. This procedure
corresponds to Example~\ref{seceleanu-example-ci}. It can be
applied to any vanishing ideal $I$ to obtain the entries of the
matrix $({\rm fp}_{I}(d,r))$ and is reasonably fast.
\begin{verbatim}
S=QQ[t1,t2,t3], I=ideal(t1^3,t2*t3)
M=coker gens gb I
regularity M, degree M, init=ideal(leadTerm gens gb I)
er=(x)-> if not quotient(init,x)==init then degree ideal(init,x) else 0
fpr=(d,r)->degree M - max apply(apply(apply(
subsets(flatten entries basis(d,M),r),toSequence),ideal),er)
hilbertFunction(1,M),fpr(1,1),fpr(1,2),fpr(1,3)
--gives the first row of the footprint matrix
\end{verbatim}
\end{procedure}


\begin{procedure}\label{procedure-gmdf}
Computing the GMD function with {\it
Macaulay\/}$2$ \cite{mac2} over a finite field and computing an upper
bound over any field using products of linear forms. This procedure
corresponds to Example~\ref{seceleanu-example-determinantal}.
\begin{verbatim}
q=3,S=ZZ/3[t1,t2,t3,t4,t5,t6],I=ideal(t1*t6-t3*t4,t2*t6-t3*t5)
G=gb I, M=coker gens gb I
regularity M, degree M, init=ideal(leadTerm gens gb I)
genmd=(d,r)->degree M-max apply(apply(subsets(apply(apply(apply(
toList (set(0..q-1))^**(hilbertFunction(d,M))-
(set{0})^**(hilbertFunction(d,M)),toList),x->basis(d,M)*vector x),
z->ideal(flatten entries z)),r),ideal),
x-> if #set flatten entries mingens ideal(leadTerm gens x)==r
and not quotient(I,x)==I then degree(I+x) else 0)
hilbertFunction(1,M),fpr(1,1),fpr(1,2),fpr(1,3),fpr(1,4),fpr(1,5),fpr(1,6)
genmd(1,1), L={t1,t2,t3,t4,t5,t6}
linearforms=(d,r)->degree M - max
apply(apply(apply((subsets(apply(apply(
(subsets(L,d)),product),x-> x % G),r)),toList),ideal),
x-> if #set flatten entries mingens ideal(leadTerm gens x)==r
and not quotient(I,x)==I then degree(I+x) else 0)
linearforms(1,2),linearforms(1,3),linearforms(1,4),linearforms(1,5)
--gives upper bound for genmd
\end{verbatim}
\end{procedure}


\begin{procedure}\label{sep5-18} Computing the
minimum socle degree and the v-number of an ideal $I$ with
{Macaulay\/}$2$ \cite{mac2}. This procedure
corresponds to Example~\ref{sep5-18-example}.
\begin{verbatim}
S=QQ[t1,t2,t3,t4]
p1=ideal(t2,t3,t4),p2=ideal(t1,t3,t4),p3=ideal(t1,t2,t4),p4=ideal(t1,t2,t3)
I=intersect(ideal(t2^10,t3^9,t4^4,t2*t3*t4^3),
ideal(t1^4,t3^4,t4^3,t1*t3*t4^2),ideal(t1^4,t2^5,t4^3),
ideal(t1^3,t2^5,t3^10))
h=ideal(t1+t2+t3+t4)--regular element on S/I
J=quotient(I+h,m), regularity coker gens gb I
soc=J/(I+h), degrees mingens soc
J1=quotient(I,p1), soc1=J1/I, degrees mingens soc1
\end{verbatim}
\end{procedure}


\begin{procedure}\label{sep12-18} Computing 
the v-number of a vanishing ideal $I(\mathbb{X})$, the regularity
index of $\delta_\mathbb{X}$, and the minimum
distance $\delta_\mathbb{X}(d)$ of the Reed--Muller-type code $C_\mathbb{X}(d)$ with
{Macaulay\/}$2$ \cite{mac2}. This procedure
corresponds to Example~\ref{Hiram-first-counterxample}.
\begin{verbatim}
q=3, G=ZZ/q, S=G[t3,t2,t1,MonomialOrder=>Lex]
p1=ideal(t2,t1-t3),p2=ideal(t2,t3),p3=ideal(t2,2*t1-t3)
p4=ideal(t1-t2,t3),p5=ideal(t1-t3,t2-t3), p6=ideal(2*t1-t3,2*t2-t3)
p7=ideal(t1,t2),p8=ideal(t1,t3),p9=ideal(t1,t2-t3),p10=ideal(t1,2*t2-t3)
I=intersect(p1,p2,p3,p4,p5,p6,p7,p8,p9,p10)
M=coker gens gb I, regularity M, degree M
init=ideal(leadTerm gens gb I)
genmd=(d,r)->degree M-max apply(apply(subsets(apply(apply(apply(
toList (set(0..q-1))^**(hilbertFunction(d,M))-
(set{0})^**(hilbertFunction(d,M)),toList),x->basis(d,M)*vector x),
z->ideal(flatten entries z)),r),ideal),
x-> if #set flatten entries mingens ideal(leadTerm gens x)==r
and not quotient(I,x)==I then degree(I+x) else 0)
genmd(1,1), genmd(2,1)
J1=quotient(I,p1), soc1=J1/I, degrees mingens soc1--gives 4
J7=quotient(I,p7), soc7=J7/I, degrees mingens soc7--gives 3
\end{verbatim}
\end{procedure}

\end{appendix}

\bibliographystyle{plain}

\end{document}